\documentclass[12pt]{article}
\usepackage[a4paper, total={17cm,25cm}]{geometry}
\usepackage{amsmath}
\usepackage{amsmath,amssymb,amsfonts,amsthm}
\usepackage{color, graphics}
\usepackage{multicol}

\usepackage{multirow}
\usepackage{tikz}
\usetikzlibrary{arrows,arrows.meta,patterns,calc}

\allowdisplaybreaks

\usepackage{float}
\usepackage{hyperref}

\usepackage[shortlabels]{enumitem}
\usepackage[normalem]{ulem}
\graphicspath{ {./ImagesForFigures/} }

\usepackage[alphabetic,y2k,initials]{amsrefs}

\newcommand{\E}{\mathcal{E}}
\newcommand{\C}{\mathcal{C}}

\numberwithin{equation}{section}

\newtheorem{theorem}{Theorem}[section]

\newtheorem{corollary}[theorem]{Corollary}

\theoremstyle{definition}
\newtheorem{example}[theorem]{Example}
\newtheorem{definition}[theorem]{Definition}
\newtheorem{remark}[theorem]{Remark}

 \title{Magic Billiards: the Case of Elliptical Boundaries}

\usepackage{authblk}
\author[1,3]{Vladimir Dragovi\'c}
\author[2,3,4]{Milena Radnovi\'c}
\affil[1]{\textsc{The University of Texas at Dallas, Department of Mathematical Sciences}}
\affil[2]{\textsc{The University of Sydney, School of Mathematics and Statistics}}
\affil[3]{\textsc{Mathematical Institute SANU, Belgrade}}
\affil[4]{\textsc{UNSW Sydney}}
\affil[ ]{\texttt{vladimir.dragovic@utdallas.edu,  milena.radnovic@sydney.edu.au}}

\date{}

\begin{document}

\maketitle

\begin{center}
\emph{Dedicated to Academician A.~T.~Fomenko on the occasion of his 80th anniversary.}
\end{center}

\smallskip

\begin{abstract}
In this work, we introduce a novel concept of magic billiards, which can be seen as an umbrella, unifying several well-known generalisations of mathematical billiards.  We analyse  properties of magic billiards in the case of elliptical boundaries.
We provide explicit conditions for periodicity in algebro-geometric, analytic, and polynomial forms.
A topological description of those billiards is given using Fomenko graphs.

\smallskip


\smallskip

\end{abstract}

\tableofcontents

\section{Introduction}\label{intro}

The elliptic billiard \cites{Bir,Bol1990,KT1991,Tab} is a notable example of a completely Liouville integrable system. 
Thus, the Liouville-Arnold theorem implies that the phase space of that particular billiard is foliated into invariant Liouville tori \cite{Ar}. 

Fomenko and his school developed a beautiful theory for topological description and classification of integrable systems using what are now known as \emph{Fomenko graphs} and \emph{Fomenko-Zieschang invariants}, see \cites{Fomenko1987,FZ1991} and, in particular, the book \cite{BF2004}, with the fundamentals of that theory, including a large list of well-known integrable systems, such as the integrable cases of rigid body motion and geodesic flows on surfaces.

The use of topological tools in the study of integrable billiards was initiated by the authors in \cites{DR2009}, see also \cites{DR2010, DR2011}. 
Further details and applications to other integrable systems can be found in the literature related to billiards \cites{Fokicheva2014, R2015, DR2017, VK2018,FV2019, FV2019a, PRK2020, DGR2021, DGR2022,BF2024}.
For the applications in the broader theory of Hamiltonian systems with two degrees of freedom see \cites{BMF1990,RRK2008,BBM2010}.

An important milestone of this theory is the so-called \emph{Fomenko conjecture}, emphasizing a surprising universality of billiard dynamics. This conjecture is about  realization of topology of Liouville foliations of smooth and real-analytic integrable Hamiltonian systems by integrable billiards, see e.g. \cite{FKK2020} and \cite{FoVed} and references therein.

In this paper, we introduce a general concept of \emph{magic billiards}, where after hitting the boundary the particle is magically transported to another point of the boundary and continues motion from there.
A formal definition of such class of systems is given in
Section \ref{sec:definition}. 
Magic billiards can be seen as an umbrella, unifying several well-known generalisations of mathematical billiards (Example \ref{ex:standard}), see e.g. Examples \ref{ex:Finsler} and \ref{ex:slipping}.

In Section \ref{sec:magic-ell}, we focus to magic billiards within an ellipse, and among them only to those ones where the equations of motion in elliptic coordinates remain the same as for the standard billiard, which will mean that the obtained system will be integrable.
We provide conditions for periodicity of such systems and give topological description using Fomenko graphs.
We note that one of the cases we consider, so called \emph{billiards with slipping} was recently introduced and studied by Fomenko and his school \cites{FVZ2021,FV2021,VZ2022,Fom2023,Zav2023}.
In Section \ref{sec:magic-ann}, we consider magic billiard within elliptic annulus. 
The last Section \ref{sec:conclusions} contains discussion.

\section{Definition of magic billiards}\label{sec:definition}

In this section, we will introduce a new class of dynamics, where a particle moves along straight segments by constant speed within a given domain in the plane, and when it reaches its boundary, it is transported to another point of the domain boundary from where it continues the motion within the domain.

More formally, we will introduce that dynamics as follows.

Let $D$ be a given domain in the plane, which is bounded by a smooth closed curve. Suppose that $\varphi$ is a continuous bijective mapping of the boundary $\partial D$ onto itself.
Now, consider the circle bundle $S^1(\partial D)$ with the base $\partial D$, such that the fiber over any point $p\in\partial D$ consists of the unit vectors in the tangent space to the plane at $p$.
Let $\varphi^*$ be a continuous bijective mapping of $S^1(\partial D)$ onto itself, satisfying the following:
\begin{itemize}
	\item $\pi\circ\varphi^*=\varphi\circ\pi$, where 
	$\pi:S^1(\partial D)\to\partial D$ is the projection to the base points;
	\item $\varphi^*$ maps tangent vectors to the boundary $\partial D$ to tangent vectors to that boundary curve;
	\item vectors pointing outwards $D$ are mapped to vectors pointing inwards $D$.
\end{itemize}

\begin{definition}
\emph{A magic billiard} $(D,\varphi,\varphi^*)$ is a dynamical system where a particle moves with a unit speed along straight lines in the interior of $D$, and when it hits the boundary at a point $A$ with velocity $\vec v$, it will bounce off at point $\varphi(A)$ with velocity $\varphi^*(\vec{v})$.
\end{definition}

\begin{example}\label{ex:standard}
Notice that the standard billiard in $D$ belongs to the class of just defined magic billiards.
There, mapping $\varphi$ is the identity and $\varphi^*$ is the billiard reflection, i.e.~reflection with respect to the direction of the tangent line to the boundary at each point.
\end{example}

\begin{example}\label{ex:Finsler}
Similarly, projective, Finsler and Minkowski billiards \cite{Tab1997,GT2002,Radn2003,KT2009,DR2012,DR2013,DR2017,ADR2019}, see also \cite{GM2024}, are classes of magic billiards with $\varphi$ being the identity map.
\end{example}

\begin{example}\label{ex:slipping}
Another class of magic billiards are billiards with slipping \cite{FVZ2021} (see also \cite{DGKh2024}).
Namely, it that case, $\varphi$ is an isometry of the boundary $\partial D$, while $\varphi^*$ is defined as follows.

First, choose a direction of the boundary $\partial D$ and note that the isometry $\varphi$ either preserves or reverses it.
If the direction is preserved, then we set that $\varphi^*$ maps 
the vectors tangent to $\partial D$ and pointing in the direction of $\partial D$ to the tangent vectors also pointing in the direction of the boundary.
If the direction if reversed, then such vectors are mapped to the vectors pointing in the opposite direction.

Now, suppose that $\vec{v}$ belongs to the fiber over point $p\in\partial D$ of the circle bundle $S^1(\partial D)$.
Denote by $\vec t$ one of the two unit tangent vectors to $\partial D$ at $p$. 
Then $\varphi^*(\vec v)$ is the unique vector satisfying the following equality of oriented angles 
$\angle(\vec t,\vec v)=2\pi-\angle(\varphi^*(\vec t),\varphi^*(\vec v))$.

Note that for general magic billiards, it is not required that mapping $\varphi$ is an isometry.
\end{example}

\section{Magic billiards within an ellipse}
\label{sec:magic-ell}

Before starting the analysis of novel examples of magic elliptical billiards, we will review the standard billiard within an ellipse.

Suppose that the ellipse is given by:
\begin{equation}\label{eq:ellipse}
\E\ :\ \frac{x^2}a+\frac{y^2}b=1, 
\quad
a>b>0.
\end{equation}
Following classical ideas of Jacobi, one can introduce \emph{the elliptic coordinates} $(\lambda_1,\lambda_2)$, which are, for each given point in the plane, the parameters of an ellipse and a hyperbola from the confocal family:
\begin{equation}\label{eq:confocal}
	\C_{\lambda}\ :\ \frac{x^2}{a-\lambda}+\frac{y^2}{b-\lambda}=1,
\end{equation}
which intersect at that point.
Note that each billiard trajectory within $\E$ has a unique \emph{caustic} $\C_{\alpha}$, which is touching each segment of the trajectory.
The differential equation of the billiard motion is separated in elliptic coordinates:
\begin{equation}\label{eq:elliptic-dif}
\frac{d\lambda_1}{\sqrt{(a-\lambda_1)(b-\lambda_1)(\alpha-\lambda_1)}}
+
\frac{d\lambda_2}{\sqrt{(a-\lambda_2)(b-\lambda_2)(\alpha-\lambda_2)}}
=
0.
\end{equation}

We are interested in magic billiards which keep those nice geometric and analytic properties of standard elliptic billiards, so we will focus to the mappings $\varphi$ and $\varphi^*$ which preserve the equation \eqref{eq:elliptic-dif}.

In particular, we note that the elliptic coordinates remain unchanged for $\varphi$ being one of the following:
\begin{itemize}
	\item the reflection with respect to one of the axes of the billiard boundary $\E$; or
	\item the half-turn around the center of $\E$.
\end{itemize}
Notice that such $\varphi$ is defined as an isometry on the whole plane, not only on the boundary of the billiard table.
Thus, the differential map $d\varphi$ will preserve the circle bundle $S^1(\partial D)$: in fact it maps each fiber of $S^1(\partial D)$ isometrically to another fiber.
Moreover, because the elliptic coordinates are invariant with respect to reflections with respect to the axes, the equation \eqref{eq:elliptic-dif} will also be invariant if $\varphi^*$ is the composition of the billiard reflection at the point of impact and the differential map $d\varphi$.
Thus, for such $\varphi$ and $\varphi^*$, the dynamics of the magic billiard $(\E,\varphi,\varphi^*)$ in elliptic coordinates will be identical to the usual billiard motion.
In particular, each trajectory of such magic billiards will have a unique caustic from the family of conics which are confocal with the boundary $\E$.

\begin{remark}\label{rem:configuration}
Note that for such $\varphi$, the configuration space can be defined as $\mathcal{D}/\sim$, where $\mathcal{D}$ is the billiard table, and $x\sim\varphi(x)$, for $x\in\E=\partial\mathcal{D}$.
The phase space is defined, analogously, by identifying the velocity vectors mapped into each other by $\varphi^*$.
\end{remark}

\begin{remark}\label{rem:every2}
Since $\varphi$ and $\varphi^*$ are involutions, one can see that every second segment of a trajectory of the corresponding magic billiard coincides with every second segment of the standard billiard, as illustrated in Figure \ref{fig:every2}.
\end{remark}
\begin{figure}[h]
	\centering
	\begin{tikzpicture}[>=Stealth]
		\coordinate (A1) at (0.2995,1.99001);
		\coordinate (A2) at (-2.97804,0.241574);
		\coordinate (A2') at (-2.97804,-0.241574);
		\coordinate (A3) at (-1.14631,1.84824);
		\coordinate (A3') at (-1.14631,-1.84824);
		\coordinate (A4) at (2.88779,-0.541882);
		\coordinate (A4') at (2.88779,0.541882);
		\coordinate (A5) at (1.84042,-1.57943);
		\coordinate (A5') at (1.84042,1.57943);
		\coordinate (A6) at (-2.70702,0.862046);
		\coordinate (A6') at (-2.70702,-0.862046);

		\draw [very thick,color=gray] (0.,0.) ellipse (3 and 2);
		\draw [thick,color=gray] (0.,0.) ellipse (2.54951 and 1.22474);

		\draw [->,line width=1.5pt] (A1)-- (A2);
		\draw [->,line width=1.5pt] (A2)--(A3');
		\draw [->,line width=1.5pt] (A3')--(A4);
		\draw [->,line width=1.5pt] (A4)--(A5');
		
		\draw[black, fill=black]
		(A1) circle (2pt);
		\draw[black, fill=black]
		(A2) circle (2pt) node[left] {$1$};
		\draw[black, fill=black]
		(A3') circle (2pt) node[below] {$2$};
		\draw[black, fill=black]
		(A4) circle (2pt) node[right] {$3$};
		\draw[black, fill=black]
		(A5') circle (2pt) node[above] {$4$};
		
		\begin{scope}[shift={(8,0)}]
		\coordinate (A1) at (0.2995,1.99001);
\coordinate (A2) at (-2.97804,0.241574);
\coordinate (A2') at (-2.97804,-0.241574);
\coordinate (A3) at (-1.14631,1.84824);
\coordinate (A3') at (-1.14631,-1.84824);
\coordinate (A4) at (2.88779,-0.541882);
\coordinate (A4') at (2.88779,0.541882);
\coordinate (A5) at (1.84042,-1.57943);
\coordinate (A5') at (1.84042,1.57943);

\draw [very thick,color=gray] (0.,0.) ellipse (3 and 2);
\draw [thick,color=gray] (0.,0.) ellipse (2.54951 and 1.22474);

\draw [->,line width=1.5pt] (A1)-- (A2);
\draw [->,line width=1.5pt] (A2')--(A3);
\draw [->,gray,dashed,line width=1.5pt] (A2)--(A3');
\draw [->,line width=1.5pt] (A3')--(A4);
\draw [->,line width=1.5pt] (A4')--(A5);
\draw [->,gray,dashed,line width=1.5pt] (A4)--(A5');

\draw[black, fill=black]
(A1) circle (2pt);
\draw[black, fill=black]
(A2) circle (2pt) node[left] {$1$};
\draw[black, fill=black]
(A2') circle (2pt);
\draw[black, fill=black]
(A3) circle (2pt) node[above] {$2$};
\draw[black, fill=black]
(A3') circle (2pt);
\draw[black, fill=black]
(A4) circle (2pt) node[right] {$3$};
\draw[black, fill=black]
(A4') circle (2pt);
\draw[black, fill=black]
(A5) circle (2pt) node[below] {$4$};
\draw[black, fill=black]
(A5') circle (2pt);
		\end{scope}
		
		\begin{scope}[shift={(0,-5)}]
	\coordinate (A1) at (0.2995,1.99001);
	\coordinate (A2) at (-2.97804,0.241574);
	\coordinate (A2') at (2.97804,0.241574);
	\coordinate (A3) at (1.14631,-1.84824);
	\coordinate (A3') at (-1.14631,-1.84824);
	\coordinate (A4) at (2.88779,-0.541882);
	\coordinate (A4') at (-2.88779,-0.541882);
	\coordinate (A5) at (-1.84042,1.57943);
	\coordinate (A5') at (1.84042,1.57943);

	\draw [very thick,color=gray] (0.,0.) ellipse (3 and 2);
	\draw [thick,color=gray] (0.,0.) ellipse (2.54951 and 1.22474);

	\draw [->,line width=1.5pt] (A1)-- (A2);
	\draw [->,line width=1.5pt] (A2')--(A3);
	\draw [->,gray,dashed,line width=1.5pt] (A2)--(A3');
	\draw [->,line width=1.5pt] (A3')--(A4);
	\draw [->,line width=1.5pt] (A4')--(A5);
	\draw [->,gray,dashed,line width=1.5pt] (A4)--(A5');

	\draw[black, fill=black]
	(A1) circle (2pt);
	\draw[black, fill=black]
	(A2) circle (2pt) node[left] {$1$};
	\draw[black, fill=black]
	(A2') circle (2pt);
	\draw[black, fill=black]
	(A3) circle (2pt) node[below] {$2$};
	\draw[black, fill=black]
	(A3') circle (2pt);
	\draw[black, fill=black]
	(A4) circle (2pt) node[right] {$3$};
	\draw[black, fill=black]
	(A4') circle (2pt);
	\draw[black, fill=black]
	(A5) circle (2pt) node[above] {$4$};
	\draw[black, fill=black]
	(A5') circle (2pt);
\end{scope}

		\begin{scope}[shift={(8,-5)}]
	\coordinate (A1) at (0.2995,1.99001);
	\coordinate (A2) at (-2.97804,0.241574);
	\coordinate (A2') at (2.97804,-0.241574);
	\coordinate (A3) at (1.14631,1.84824);
	\coordinate (A3') at (-1.14631,-1.84824);
	\coordinate (A4) at (2.88779,-0.541882);
	\coordinate (A4') at (-2.88779,0.541882);
	\coordinate (A5) at (-1.84042,-1.57943);
	\coordinate (A5') at (1.84042,1.57943);
	\coordinate (A2a) at ($ (A1)!8.0/10!(A2) $);

	\draw [very thick,color=gray] (0.,0.) ellipse (3 and 2);
	\draw [thick,color=gray] (0.,0.) ellipse (2.54951 and 1.22474);

	\draw [->,line width=1.5pt] (A1)-- (A2a);
	\draw [line width=1.5pt] (A1)-- (A2);
	\draw [->,line width=1.5pt] (A2')--(A3);
	\draw [->,gray,dashed,line width=1.5pt] (A2)--(A3');
	\draw [->,line width=1.5pt] (A3')--(A4);
	\draw [->,line width=1.5pt] (A4')--(A5);
	\draw [->,gray,dashed,line width=1.5pt] (A4)--(A5');

	\draw[black, fill=black]
	(A1) circle (2pt);
	\draw[black, fill=black]
	(A2) circle (2pt) node[left] {$1$};
	\draw[black, fill=black]
	(A2') circle (2pt);
	\draw[black, fill=black]
	(A3) circle (2pt) node[above] {$2$};
	\draw[black, fill=black]
	(A3') circle (2pt);
	\draw[black, fill=black]
	(A4) circle (2pt) node[below right] {$3$};
	\draw[black, fill=black]
	(A4') circle (2pt);
	\draw[black, fill=black]
	(A5) circle (2pt) node[below] {$4$};
	\draw[black, fill=black]
	(A5') circle (2pt);
\end{scope}

	\end{tikzpicture}
	\caption{A standard billiard trajectory (upper left) and three trajectories of the magic billiard with the same initial segment.
	Upper right: flipping over the long axis; lower left: flipping over the short axis; lower right: half-turn around the center.
	In each case, the consecutive points of impact with the boundary are numerated.
	}	\label{fig:every2}
\end{figure}
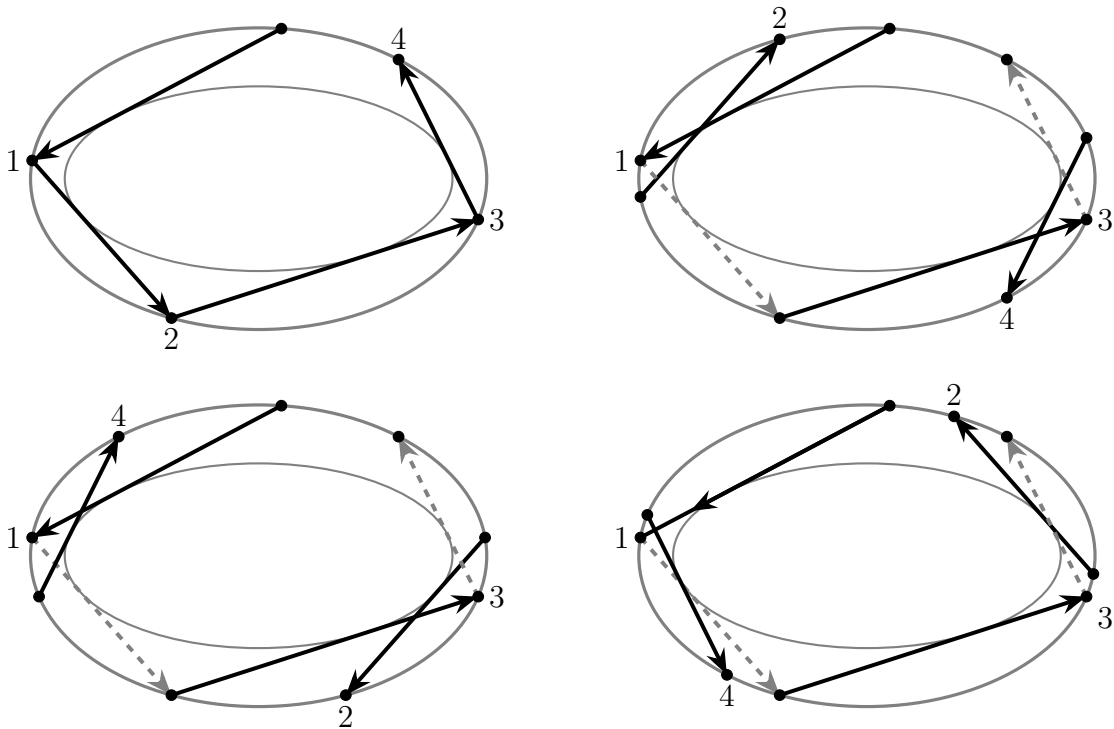

\begin{remark}\label{rem:even}
As a consequence of Remark \ref{rem:every2}, we note that a magic billiard trajectory in those cases is periodic if and only if the corresponding standard billiard trajectory is periodic.
	Moreover,
 the $n$-periodicity conditions for even $n$ will be the same as for the standard elliptical billiard.
For odd periods, the conditions will be different, and will be distinct for each case that we consider. 
\end{remark}

\begin{remark}
Note that the construction of the magic billiard trajectories reminds of the construction of the standard billiard trajectories within a half-ellipse or a quarter-ellipse.
Namely, in all of those cases any trajectory in part coincides with the trajectory of elliptic billiard with the same initial conditions, while in part is symmetric to that trajectory, as illustrated in Figures \ref{fig:every2} and \ref{fig:half}.
Note that the trajectories of each magic billiard that we consider and the standard billiards within ellipse, half-ellipse, and quarter-ellipse satisfy the equation \eqref{eq:elliptic-dif}, thus they all look the same in elliptic coordinates.
\end{remark}
\begin{figure}[h]
	\centering
	\begin{tikzpicture}[>=Stealth]
		\coordinate (A1) at (0.2995,1.99001);
		\coordinate (A2) at (-2.97804,0.241574);
		\coordinate (A2') at (-2.97804,-0.241574);
		\coordinate (A3) at (-1.14631,1.84824);
		\coordinate (A3') at (-1.14631,-1.84824);
		\coordinate (A4) at (2.88779,-0.541882);
		\coordinate (A4') at (2.88779,0.541882);
		\coordinate (A5) at (1.84042,-1.57943);
		\coordinate (A5') at (1.84042,1.57943);
		\coordinate (A6) at (-2.70702,0.862046);
		\coordinate (A6') at (-2.70702,-0.862046);

		\draw [very thick,color=black] (0.,0.) ellipse (3 and 2);
		\draw [thick,color=gray] (0.,0.) ellipse (2.54951 and 1.22474);

		\draw [->,line width=1.5pt] (A1)-- (A2)--(A3')--(A4)--(A5');
		
		\draw[black, fill=black]
		(A1) circle (2pt);
		\draw[black, fill=black]
		(A2) circle (2pt);
		\draw[black, fill=black]
		(A3') circle (2pt);
		\draw[black, fill=black]
		(A4) circle (2pt);
		\draw[black, fill=black]
		(A5') circle (2pt);
		
		\begin{scope}[shift={(8,0)}]
			\coordinate (A1) at (0.2995,1.99001);
			\coordinate (A2) at (-2.97804,0.241574);
			\coordinate (A2') at (-2.97804,-0.241574);
			\coordinate (A3) at (-1.14631,1.84824);
			\coordinate (A3') at (-1.14631,-1.84824);
			\coordinate (A4) at (2.88779,-0.541882);
			\coordinate (A4') at (2.88779,0.541882);
			\coordinate (A5) at (1.84042,-1.57943);
			\coordinate (A5') at (1.84042,1.57943);
			\coordinate (A1h) at (-2.7663,0);
			\coordinate (A4h) at (2.62024,0);
			
			\draw [very thick,dashed,color=gray] (-3,0.) arc (-180:0:3 and 2);
			\draw [very thick,color=black] (3,0.) arc (0:180:3 and 2);
			\draw[very thick,color=black](-3,0)--(3,0);
			\draw [thick,color=gray] (0,0.) ellipse (2.54951 and 1.22474);

			\draw [->,line width=1.5pt] (A1)-- (A2)--(A1h)--(A3)--(A4')--(A4h)--(A5');
			\draw [gray,dashed,line width=1.5pt] (A1h)--(A3')--(A4)--(A4h);

			\draw[black, fill=black]
			(A1) circle (2pt);
			\draw[black, fill=black]
			(A1h) circle (2pt);
			\draw[black, fill=black]
			(A4h) circle (2pt);
			
			\draw[black, fill=black]
			(A2) circle (2pt);
			\draw[black, fill=black]
			(A3) circle (2pt);
			\draw[gray, fill=gray]
			(A3') circle (2pt);
			
			\draw[gray, fill=gray]
			(A4) circle (2pt);
			\draw[black, fill=black]
			(A4') circle (2pt);
			\draw[black, fill=black]
			(A5') circle (2pt);
		\end{scope}
		
		\begin{scope}[shift={(0,-5)}]
			\coordinate (A1) at (0.2995,1.99001);
			\coordinate (A2) at (-2.97804,0.241574);
			\coordinate (A2') at (2.97804,0.241574);
			\coordinate (A3) at (1.14631,-1.84824);
			\coordinate (A3') at (-1.14631,-1.84824);
			\coordinate (A4) at (2.88779,-0.541882);
			\coordinate (A4') at (-2.88779,-0.541882);
			\coordinate (A5) at (-1.84042,1.57943);
			\coordinate (A5') at (1.84042,1.57943);
			\coordinate (A1v) at (0,1.83024);
			\coordinate (A3v) at (0,-1.47703);

			\draw [very thick,dashed,color=gray] (0,2) arc (90:270:3 and 2);
			\draw [very thick,color=black] (0,-2) arc (-90:90:3 and 2);
			\draw [thick,color=gray] (0,0) ellipse (2.54951 and 1.22474);
			\draw[very thick,color=black](0,-2)--(0,2);

			\draw [->,line width=1.5pt] (A1)--(A1v)--(A2')--(A3)--(A3v)--(A4)--(A5');
			\draw [gray,dashed,line width=1.5pt] (A1v)--(A2)--(A3')--(A3v);

			\draw[black, fill=black]
			(A1) circle (2pt);
			\draw[gray, fill=gray]
			(A2) circle (2pt);
			\draw[black, fill=black]
			(A3) circle (2pt);
			\draw[gray, fill=gray]
			(A3') circle (2pt);
			\draw[black, fill=black]
			(A4) circle (2pt);
			\draw[black, fill=black]
			(A5') circle (2pt);
			\draw[black, fill=black]
			(A1v) circle (2pt);
			\draw[black, fill=black]
			(A3v) circle (2pt);
		\end{scope}
		
		\begin{scope}[shift={(8,-5)}]
			\coordinate (A1) at (0.2995,1.99001);
			\coordinate (A2) at (-2.97804,0.241574);
			\coordinate (A2') at (2.97804,0.241574);
			\coordinate (A3) at (1.14631,1.84824);
			\coordinate (A3') at (-1.14631,-1.84824);
			\coordinate (A4) at (2.88779,-0.541882);
			\coordinate (A4') at (2.88779,0.541882);
			\coordinate (A5) at (1.84042,-1.57943);
			\coordinate (A5') at (1.84042,1.57943);
			\coordinate (A2a) at ($ (A1)!8.0/10!(A2) $);
			\coordinate (A1h) at (2.7663,0);
			\coordinate (A4h) at (2.62024,0);
			\coordinate (A1v) at (0,1.83024);
			\coordinate (A3v) at (0,1.47703);

			\draw [very thick,dashed,color=gray] (0,2) arc (90:360:3 and 2);
			\draw [very thick,color=black] (3,0) arc (0:90:3 and 2);
			\draw[very thick,color=black](3,0)--(0,0)--(0,2);
			\draw [thick,color=gray] (0.,0.) ellipse (2.54951 and 1.22474);

			\draw [->,black,line width=1.5pt] (A1)-- (A1v)--(A2')--(A1h)--(A3)--(A3v)--(A4')--(A4h)--(A5');
			\draw [gray,dashed,line width=1.5pt] (A1v)--(A2)--(A3')--(A4)--(A4h);

\draw[black, fill=black](A1) circle (2pt);
\draw[black, fill=black](A1v) circle (2pt);
\draw[black, fill=black](A2') circle (2pt);
\draw[black, fill=black](A1h) circle (2pt);
\draw[black, fill=black](A3) circle (2pt);
\draw[black, fill=black](A3v) circle (2pt);
\draw[black, fill=black](A4') circle (2pt);
\draw[black, fill=black](A4h) circle (2pt);
\draw[black, fill=black](A5') circle (2pt);

\draw[gray, fill=gray](A2) circle (2pt);
\draw[gray, fill=gray](A3') circle (2pt);
\draw[gray, fill=gray](A4) circle (2pt);

		\end{scope}
		
	\end{tikzpicture}
	\caption{Billiards within ellipse (upper left), half-ellipses (upper right and lower left), and quarter-ellipse (lower right). 
	All trajectories have the same initial conditions.
	Unfolding the trajectories within the last three domains about the axes give the trajectory within the ellipse.
	}	\label{fig:half}
\end{figure}
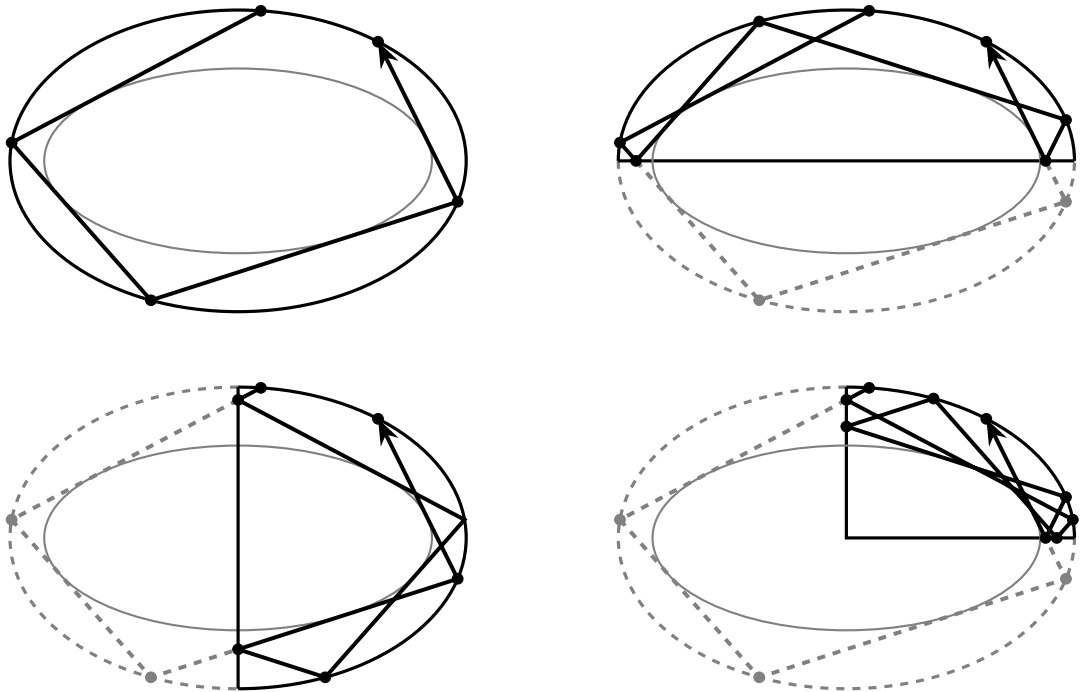

\begin{remark}\label{rem:flip}
In \cite{Fokicheva2015}, generalized billiard domains obtained by gluing the pieces bounded by confocal conics along parts of their boundaries were classified and the Fomenko-Zieschang invariants were calculated for some such domains. In particular,
the billiards with flipping over one of the axes can be related to topological billiards $\Delta_{\beta}(A_2')^2_{2x}$ and $\Delta_{\beta}(A_1)^2_{2y}$ from \cite{Fokicheva2015}, by
reflecting half of the table relative to the corresponding axis and then gluing
the two halves of the disk along the isometry of the elliptic boundary arcs.
\end{remark}

In the following subsections, we will consider each of the three magic billiards separately.
We will derive periodicity conditions in three forms: algebro-geometric, analytic, and polynomial.
We will also present Fomenko graphs corresponding to the Liouville foliation of the phase space.

\subsection{
	Elliptical billiard with flipping over long axis}

In this system, $\varphi$ is the reflection with respect to the longer axis.
Examples of its trajectories are depicted in Figure \ref{fig:magic1}.	
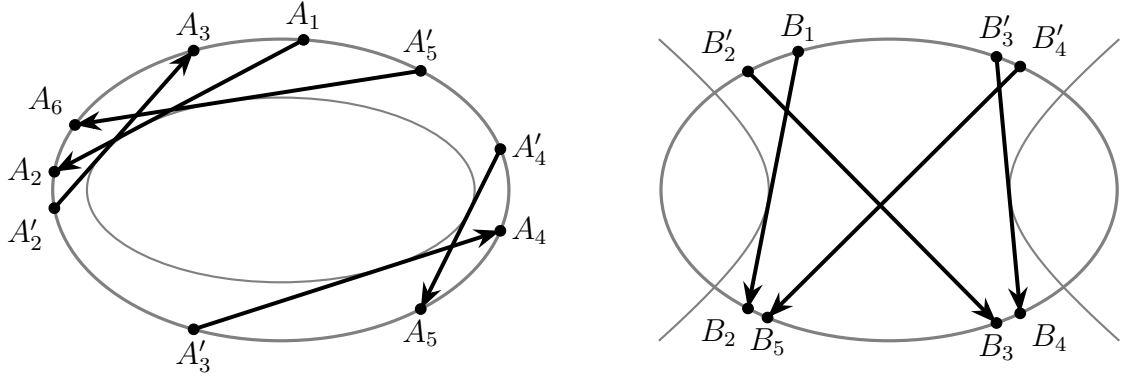
\begin{figure}[h]
	\centering
	\begin{tikzpicture}[>=Stealth]
		\coordinate (A1) at (0.2995,1.99001);
		\coordinate (A2) at (-2.97804,0.241574);
		\coordinate (A2') at (-2.97804,-0.241574);
		\coordinate (A3) at (-1.14631,1.84824);
		\coordinate (A3') at (-1.14631,-1.84824);
		\coordinate (A4) at (2.88779,-0.541882);
		\coordinate (A4') at (2.88779,0.541882);
		\coordinate (A5) at (1.84042,-1.57943);
		\coordinate (A5') at (1.84042,1.57943);
		\coordinate (A6) at (-2.70702,0.862046);
		\coordinate (A6') at (-2.70702,-0.862046);

		\draw [very thick,color=gray] (0.,0.) ellipse (3 and 2);
		\draw [thick,color=gray] (0.,0.) ellipse (2.54951 and 1.22474);

		\draw [->,line width=1.5pt] (A1)-- (A2);
		\draw [->,line width=1.5pt] (A2')--(A3);
		\draw [->,line width=1.5pt] (A3')--(A4);
		\draw [->,line width=1.5pt] (A4')--(A5);
		\draw [->,line width=1.5pt] (A5')--(A6);
	
	\draw[black, fill=black]
	(A1) circle (2pt) node[above] {$A_1$};
	\draw[black, fill=black]
	(A2) circle (2pt) node[left] {$A_2$};
	\draw[black, fill=black]
	(A2') circle (2pt) node[below left] {$A_2'$};
		\draw[black, fill=black]
	(A3) circle (2pt) node[above] {$A_3$};
	\draw[black, fill=black]
	(A3') circle (2pt) node[below] {$A_3'$};
		\draw[black, fill=black]
	(A4) circle (2pt) node[right] {$A_4$};
	\draw[black, fill=black]
	(A4') circle (2pt) node[right] {$A_4'$};
		\draw[black, fill=black]
	(A5) circle (2pt) node[below] {$A_5$};
	\draw[black, fill=black]
	(A5') circle (2pt) node[above] {$A_5'$};
	\draw[black, fill=black]
	(A6) circle (2pt) node[above left] {$A_6$};
	
	\begin{scope}[shift={(8,0)}]
			\coordinate (B1) at (-1.19522, 1.83442);
		\coordinate (B2) at (-1.85594, -1.57134);
		\coordinate (B2') at (-1.85594, 1.57134);
		\coordinate (B3) at (1.41241, -1.76448);
		\coordinate (B3') at (1.41241, 1.76448);
		\coordinate (B4) at (1.72475, -1.63643);
		\coordinate (B4') at (1.72475, 1.63643);
		\coordinate (B5) at (-1.59857, -1.69241);
		
			\draw [very thick,color=gray] (0.,0.) ellipse (3 and 2);
	
	\draw[thick, gray, domain=-2:2,smooth,variable=\y]
	plot ({sqrt(2.5+5/3*\y*\y)},{\y});
	\draw[thick, gray, domain=-2:2,smooth,variable=\y]
	plot ({-sqrt(2.5+5/3*\y*\y)},{\y});

	\draw [->,line width=1.5pt] (B1)-- (B2);
	\draw [->,line width=1.5pt] (B2')--(B3);
	\draw [->,line width=1.5pt] (B3')--(B4);
	\draw [->,line width=1.5pt] (B4')--(B5);
	
	\draw[black, fill=black]
	(B1) circle (2pt) node[above] {$B_1$};
	\draw[black, fill=black]	(B2) circle (2pt) node[below left] {$B_2$};
	\draw[black, fill=black](B2') circle (2pt) node[above left] {$B_2'$};
	\draw[black, fill=black](B3) circle (2pt) node[below] {$B_3$};
	\draw[black, fill=black](B3') circle (2pt) node[above] {$B_3'$};
	\draw[black, fill=black](B4) circle (2pt) node[below right] {$B_4$};
	\draw[black, fill=black](B4') circle (2pt) node[above right] {$B_4'$};
	\draw[black, fill=black](B5) circle (2pt) node[below] {$B_5$};
	\end{scope}

\end{tikzpicture}
	\caption{Magic billiard with flipping over the long axis: each time when the particle hits the boundary, it is reflected off the boundary and immediately magically flipped over the longer axis.
	}	\label{fig:magic1}
\end{figure}

Periodicity conditions for such trajectories can be obtained by applying the results from \cite{DR2019}.
In the following theorem, we represent such conditions in the algebro-geometric form, from which the classical form of Cayley's and in the polynomial form can be derived.

\begin{theorem}\label{th:divisor-magic1}
In the billiard table bounded by the ellipse $\E$ given by \eqref{eq:ellipse}, consider the magic billiard with flipping over the long axis.
Consider a trajectory of such billiard with the caustic $\C_{\beta}$ from the confocal family \eqref{eq:confocal}.
Such a trajectory
is $n$-periodic if and only if one of the following conditions is satisfied:
	\begin{itemize}
		\item $n$ is even and $nQ_0\sim nQ_{\infty}$;
		\item $n$ is odd, $\C_{\beta}$ is hyperbola, and $nQ_0\sim nQ_b$. 
	\end{itemize}
Here $Q_0$, $Q_b$, $Q_{\infty}$ denote the points with coordinates $(0,\sqrt{ab\beta})$, $(b,0)$, $(\infty,\infty)$ on the elliptic curve:
\begin{equation}\label{eq:curve}
	y^2=(a-x)(b-x)(\beta-x).
\end{equation}
\end{theorem}
\begin{proof}	
	Integrating \eqref{eq:elliptic-dif} along a trajectory of the magic billiard, we get that the $n$-periodicity condition will be equivalent to the following divisor condition:
	\begin{equation}\label{eq:divisor-cayley}
		n(Q_0-Q_{c_1})+m(Q_{c_2}-Q_a)\sim0
	\end{equation}
	on the elliptic curve \eqref{eq:curve}.
	Here, we have $\{c_1,c_2\}=\{\beta,b\}$ and $c_1<c_2$, while $Q_{\beta}$, $Q_a$, denote the points of the curve with coordinates $(\beta,0)$, $(a,0)$ respectively.
	Note that, on the trajectory, the elliptic coordinates $\lambda_1$ and $\lambda_2$ belong to the intervals $[0,c_1]$ and $[c_2,a]$ respectively.
	Natural numbers $n$ and $m$ represent the numbers of times each of them traces back and forth its respective interval along the closed trajectory.
	
	Note that the short axis of the ellipse is the degenerate conic $\C_a$ of the confocal family.
	Since that axis must be crossed even number of times along a closed trajectory, we have that $m$ is even, thus the periodicity condition \eqref{eq:divisor-cayley} reduces to $n(Q_0-Q_{c_1})\sim0$.
	
	On the other hand, the sum of the number of the flippings over the long axis and the number of times when the billiard particle actually crossed the long axis must also be even.
	That sum equals $n+m$ if $\C_{\beta}$ is an ellipse, and $2n$ if $\C_{\beta}$ is hyperbola.
	From there, we conclude that all closed trajectories with ellipse as caustic must have even period.
	
Now, the required conditions follow from the fact that $c_1=b$ when $\mathcal{C}_{\beta}$ is hyperbola and $2Q_{\beta}\sim2Q_{\infty}$.
\end{proof}

Now, analogously as in \cite{DR2019}, we can get the conditions of Cayley's type.

\begin{corollary}\label{cor:cayley-magic1}
A trajectory of the billiard within $\E$ with flipping over the long axis is $n$-periodic if and only if one of the following conditions is satisfied:
\begin{itemize}
	\item $n$ is even and
\begin{equation}\label{eq:cayley-even}
	\left|	\begin{array}{llll}
		B_3 & B_4 & \dots & B_{n/2+1}
		\\
		B_4 & B_5 & \dots & B_{n/2+2}
		\\
		\dots\\
		B_{n/2+1} & B_{n/2+2} & \dots & B_{n-1}
	\end{array}
	\right|=0;
	\end{equation}
	\item $n$ is odd, the caustic is hyperbola, and 
	\begin{equation}\label{eq:cayley-odd}
	\left|	\begin{array}{llll}
		C_2 & C_3 & \dots & C_{(n+1)/2}
		\\
		C_3 & C_4 & \dots & C_{(n+1)/2+1}
		\\
		\dots\\
		C_{(n+1)/2} & C_{(n+1)/2+1} & \dots & C_{n-1}
	\end{array}
	\right|=0.
	\end{equation}
 \end{itemize}
	Here, we denoted:
\begin{gather*}
	\sqrt{(a-x)(b-x)(\beta-x)}=B_0+B_1x+B_2x^2+\dots,
	\\
	\frac{\sqrt{(a-x)(b-x)(\beta-x)}}{b-x}=C_0+C_1x+C_2x^2+\dots,
\end{gather*}
the Taylor expansions around $x=0$, and $\beta$ is the parameter of the caustic from \eqref{eq:confocal}.	
\end{corollary}

The periodicity conditions can be equivalently stated in the form of polynomial equations.
Again, for details on how to obtain those equations, refer to \cite{DR2019}.
A more general theory connecting polynomial Pell's equations and integrable billiards in arbitrary dimension is given in \cite{DR2018}.
\begin{corollary}\label{cor:magic-pol2}
The trajectories of the magic billiard with flipping over the long axis with caustic $\C_{\beta}$ are $n$-periodic if and only if there exists a pair of real polynomials $p_{d_1}$, $q_{d_2}$ of degrees $d_1$, $d_2$ respectively, and satisfying the following:
	\begin{itemize}
		\item[(a)] if $n=2m$ is even,  then $d_1=m$, $d_2=m-2$, and
		\begin{equation}\label{eq:pol1}
		p_{m}^2(s)
		-
		s\left(s-\frac1a\right)\left(s-\frac1b\right)\left(s-\frac1{\beta}\right)
		{q}_{m-2}^2(s)=1;
		\end{equation}
		\item[(b)] if $n=2m+1$ is odd, then $d_1=m$, $d_2=m-1$, and
		$$
		\left(s-\frac1{b}\right)p_m^2(s)
		-
		s\left(s-\frac1a\right)\left(s-\frac1{\beta}\right)q_{m-1}^2(s)=-1.
		$$
	\end{itemize}	
\end{corollary}

We conclude this section by a topological description of the system.

\begin{theorem}\label{th:fomenko-magic1}
The Fomenko-Zieschang invariant of the magic billiard with flipping over the long axis is shown in Figure \ref{fig:fom-magic1}.
\end{theorem}
\begin{figure}[h]
	\centering
	\begin{tikzpicture}
		\tikzset{vertex/.style = {shape=circle,draw,minimum size=1.5em}}
		\node[vertex] (aa) at  (0,0) {$\mathbf{A}$};
		\node[vertex] (bb) at  (4,0) {$\mathbf{B}$};
		\node[vertex] (ee) at  (8,1.5) {$\mathbf{A}$};
		\node[vertex] (jj) at  (8,-1.5) {$\mathbf{A}$};
		\path[-{Stealth[length=3mm]}] 
		(bb) edge node[above]{$r=0$} 
				  node[below]{$\varepsilon=1$}	(aa) 
		(bb) edge node[above,sloped]{$r=\dfrac12$} 
		node[below,sloped]{$\varepsilon=1$}
		(ee) 
		(bb) edge node[below,sloped]{$r=\dfrac12$} 
		node[above,sloped]{$\varepsilon=1$}(jj); 
		
		\draw (0,-3) -- (8,-3);
		\foreach \x in {0,4,8}
		\draw[shift={(\x,-3)},color=black] (0pt,3pt) -- (0pt,-3pt);
		\node[below] at (0,-3) { $\lambda=0$};
		\node[below] at (4,-3) { $\lambda=b$};
		\node[below] at (8,-3) { $\lambda=a$};
		
		\draw[black, dashed]
		(bb) circle (25pt);
		\node at (4,-1.2) {$n=-2$};
	\end{tikzpicture}
	\caption{Theorem \ref{th:fomenko-magic1}: Fomenko graph for the magic billiard with flipping over the long axis.}
	\label{fig:fom-magic1}
\end{figure}
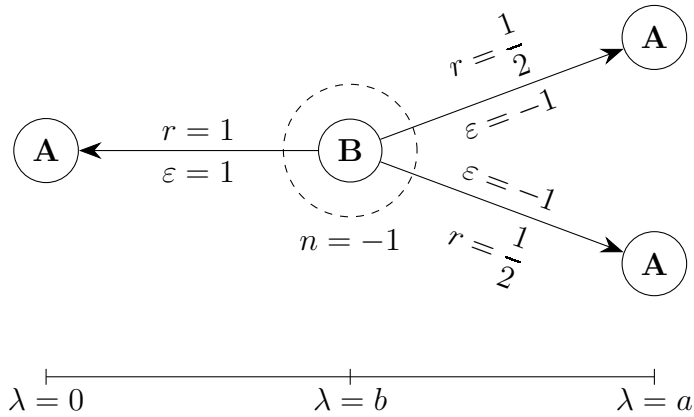
\begin{proof}
The billiard table consists of the boundary ellipse $\E$ and its interior. 
According to Remark \ref{rem:configuration}, the configuration space is obtained when the points of $\E$ which are symmetric to each other with respect to the long axis are identified.
Thus, the configuration space is homeomorphic to a sphere, where the boundary of the billiard table is represented by an arc.

Any level set corresponding to the trajectories having a fixed ellipse as caustic consists of a single Liouville torus, see Figure \ref{fig:glue1}.
	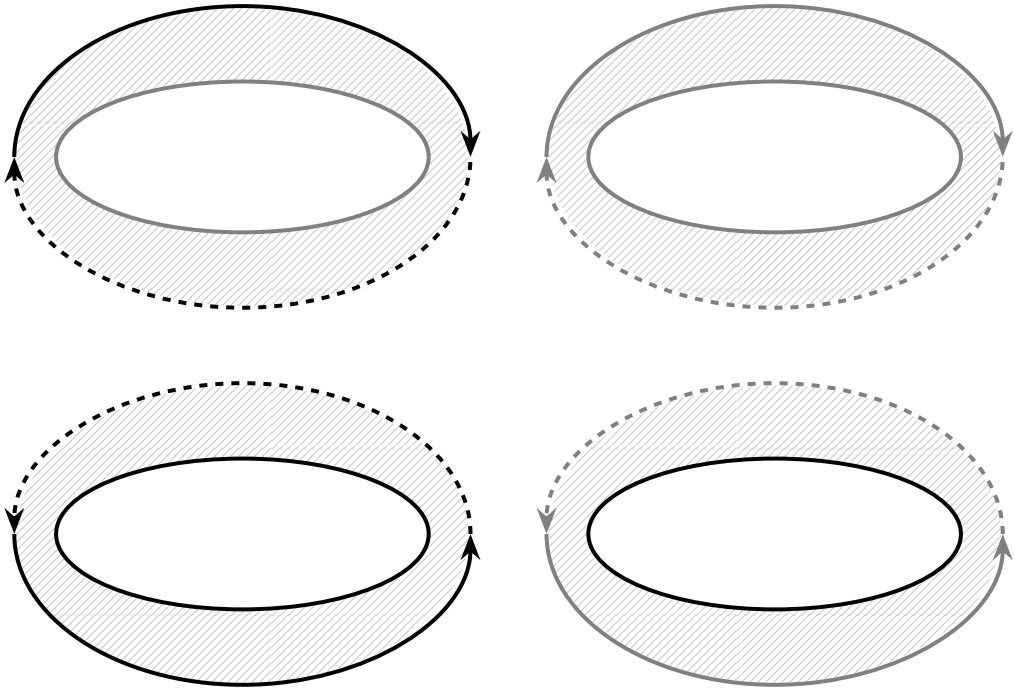
\begin{figure}[h]
	\centering
	\begin{tikzpicture}
		\draw [line width=1.5pt,color=white,pattern=north east lines,pattern color=gray!40] (0.,0.) ellipse (3.cm and 2.cm);
		
		\draw[line width=1.5pt, black, Stealth-] (3,0) arc
		(0:180:3cm and 2cm);
		
		\draw[line width=1.5pt, black, dashed, Stealth-] (-3,0) arc
		(180:360:3cm and 2cm);
		
		\draw [line width=1.5pt, 
		color=gray, fill=white, fill opacity=1.0] (0.,0.) ellipse (2.45cm and 1.cm);
		\begin{scope}[shift={(7,0)}]
			\draw [line width=1.5pt,color=white,pattern=north east lines,pattern color=gray!40] (0.,0.) ellipse (3.cm and 2.cm);
			
			\draw[line width=1.5pt, gray, Stealth-] (3,0) arc
			(0:180:3cm and 2cm);
			
			\draw[line width=1.5pt, gray, dashed, Stealth-] (-3,0) arc
			(180:360:3cm and 2cm);
			
			\draw [line width=1.5pt, 
			color=gray, fill=white, fill opacity=1.0] (0.,0.) ellipse (2.45cm and 1.cm);
		\end{scope}
		
		\begin{scope}[shift={(0,-5)}]
			\draw [line width=1.5pt,color=white,pattern=north east lines,pattern color=gray!40] (0.,0.) ellipse (3.cm and 2.cm);
			
			\draw[line width=1.5pt, black, dashed, -Stealth] (3,0) arc
			(0:180:3cm and 2cm);
			
			\draw[line width=1.5pt, black, -Stealth] (-3,0) arc
			(180:360:3cm and 2cm);
			
			\draw [line width=1.5pt, 
			color=black, fill=white, fill opacity=1.0] (0.,0.) ellipse (2.45cm and 1.cm);
		\end{scope}
		
		\begin{scope}[shift={(7,-5)}]
			\draw [line width=1.5pt,color=white,pattern=north east lines,pattern color=gray!40] (0.,0.) ellipse (3.cm and 2.cm);
			
			\draw[line width=1.5pt, gray, dashed, -Stealth] (3,0) arc
			(0:180:3cm and 2cm);
			
			\draw[line width=1.5pt, gray, -Stealth] (-3,0) arc
			(180:360:3cm and 2cm);
			
			\draw [line width=1.5pt, 
			color=black, fill=white, fill opacity=1.0] (0.,0.) ellipse (2.45cm and 1.cm);
		\end{scope}

	\end{tikzpicture}
	\caption{The Liouville torus corresponding to an ellipse as caustic of magic billiard with flipping over the long axis is obtained by gluing four annuli along congruent arcs of the same color and texture.}	\label{fig:glue1}
\end{figure}

The level set corresponding to the trajectories having a fixed hyperbola as caustic consists of two Liouville tori: one torus contains the trajectories where motion along each segment is downwards (as shown in the lefthandside of Figure \ref{fig:magic1}), and on the other torus the motion is upwards.

The level set corresponding to the caustic $\E=\C_0$ consists of the limit motion back and forth along the boundary.
This is a single closed trajectory, thus corresponding to the Fomenko atom $\mathbf{A}$.

The level set with the caustic $\C_a$ corresponds to the motion along the short axis of the ellipse $\E$.
When the particle hits the boundary, it is magically flipped to the opposite point on the axis.
Thus, there are two closed trajectories there: the downwards one and the upwards one.
Each of those trajectories corresponds to one Fomenko atom of type $\mathbf{A}$.

The level set $\C_b$ consists of the trajectories that contain the foci of the ellipse $\E$.
Exactly one of those trajectories is closed, corresponding to horizontal motion along the long axis, i.e.~the level set is of complexity $1$.
All other trajectories alternately pass through the one and the other focus, and on each such trajectory, because of the flipping along the long axis, the motion is either upwards along each segment or downwards along each segment.
Thus, there are two separatrices on that level set.

According to the classification of Fomenko atoms from \cite{BF2004}, there are three types of atoms of complexity $1$: atom $\mathbf{A}$ consisting of a single closed orbit without separatrices, atom $\mathbf{A^{*}}$ which has one separatrix, and atom $\mathbf{B}$ with two separatrices.
From there, the level set $\C_b$ must be $\mathbf{B}$, which is consistent with the number of connected components of regular level sets close to $\C_b$ -- that number is $2$ for $\lambda>b$ and $1$ for $\lambda<b$.
Thus, the rough Liouville equivalence class of this billiard is represented by the graph in Figure \ref{fig:fom-magic1}.

In order to get the numerical invariants, notice that the Liouville torus corresponding to the motion where an ellipse is a caustic is obtained by gluing four copies of the annulus between the billiard boundary and the the caustic, as shown in Figure \ref{fig:glue1}, while the two Liouville tori on the level set corresponding to a hyperbola as caustic are obtained by gluing two pairs of regions within the ellipse which is between the branches of hyperbola.
Choosing the coordinate systems on those tori as explained in \cite{BF2004},  one can calculate the numerical marks which are shown in Figure \ref{fig:fom-magic1}.
For the details of the calculation, check \cite{Fokicheva2015}, the domain $\Delta_{\beta}(A_2')^2_{2y}$.
\end{proof}

\subsection{Elliptic billiard with flipping over short axis}

In this case, with $\varphi$ is the reflection with respect to the short axis.
	Examples of its trajectories are depicted in Figure \ref{fig:magic2}.
\begin{figure}[h]
	\centering
	\begin{tikzpicture}[>=Stealth]
		\coordinate (A1) at (0.2995,1.99001);
		\coordinate (A2) at (-2.97804,0.241574);
		\coordinate (A2') at (2.97804,0.241574);
		\coordinate (A3) at (1.14631,-1.84824);
		\coordinate (A3') at (-1.14631,-1.84824);
		\coordinate (A4) at (2.88779,-0.541882);
		\coordinate (A4') at (-2.88779,-0.541882);
		\coordinate (A5) at (-1.84042,1.57943);
		\coordinate (A5') at (1.84042,1.57943);
		\coordinate (A6) at (-2.70702,0.862046);
		\coordinate (A6') at (2.70702,0.862046);

		\draw [very thick,color=gray] (0.,0.) ellipse (3 and 2);
		\draw [thick,color=gray] (0.,0.) ellipse (2.54951 and 1.22474);

		\draw [->,line width=1.5pt] (A1)-- (A2);
		\draw [->,line width=1.5pt] (A2')--(A3);
		\draw [->,line width=1.5pt] (A3')--(A4);
		\draw [->,line width=1.5pt] (A4')--(A5);
		\draw [->,line width=1.5pt] (A5')--(A6);
		
		\draw[black, fill=black]
		(A1) circle (2pt) node[above] {$A_1$};
		\draw[black, fill=black]
		(A2) circle (2pt) node[left] {$A_2$};
		\draw[black, fill=black]
		(A2') circle (2pt) node[right] {$A_2'$};
		\draw[black, fill=black]
		(A3) circle (2pt) node[below] {$A_3$};
		\draw[black, fill=black]
		(A3') circle (2pt) node[below] {$A_3'$};
		\draw[black, fill=black]
		(A4) circle (2pt) node[right] {$A_4$};
		\draw[black, fill=black]
		(A4') circle (2pt) node[left] {$A_4'$};
		\draw[black, fill=black]
		(A5) circle (2pt) node[above] {$A_5$};
		\draw[black, fill=black]
		(A5') circle (2pt) node[above] {$A_5'$};
		\draw[black, fill=black]
		(A6) circle (2pt) node[above left] {$A_6$};
		
		\begin{scope}[shift={(8,0)}]
			
			\coordinate (B1) at (-1.19522, 1.83442);
			\coordinate (B2) at (-1.85594, -1.57134);
			\coordinate (B2') at (1.85594, -1.57134);
			\coordinate (B3) at (-1.41241, 1.76448);
			\coordinate (B3') at (1.41241, 1.76448);
			\coordinate (B4) at (1.72475, -1.63643);
			\coordinate (B4') at (-1.72475, -1.63643);
			\coordinate (B5) at (1.59857, 1.69241);
			
			\coordinate (B3a) at ($ (B2')!8.0/10!(B3) $);
			\coordinate (B4a) at ($ (B3')!8.0/10!(B4) $);
			\coordinate (B5a) at ($ (B4')!8.0/10!(B5) $);

			\draw [very thick,color=gray] (0.,0.) ellipse (3 and 2);
			
			\draw[thick, gray, domain=-2:2,smooth,variable=\y]
			plot ({sqrt(2.5+5/3*\y*\y)},{\y});
			\draw[thick, gray, domain=-2:2,smooth,variable=\y]
			plot ({-sqrt(2.5+5/3*\y*\y)},{\y});

			\draw [->,line width=1.5pt] (B1)-- (B2);
			\draw [->,line width=1.5pt] (B2')--(B3a);
			\draw[line width=1.5pt](B2')--(B3);
			\draw [->,line width=1.5pt] (B3')--(B4a);
			\draw [line width=1.5pt] (B3')--(B4);
			\draw [->,line width=1.5pt] (B4')--(B5a);
			\draw [line width=1.5pt] (B4')--(B5);
			
			\draw[black, fill=black]
			(B1) circle (2pt) node[above] {$B_1$};
			\draw[black, fill=black]	(B2) circle (2pt) node[below left] {$B_2$};
			\draw[black, fill=black](B2') circle (2pt) node[below right] {$B_2'$};
			\draw[black, fill=black](B3) circle (2pt) node[above left] {$B_3$};
			\draw[black, fill=black](B3') circle (2pt) node[above] {$B_3'$};
			\draw[black, fill=black](B4) circle (2pt) node[below] {$B_4$};
			\draw[black, fill=black](B4') circle (2pt) node[below] {$B_4'$};
			\draw[black, fill=black](B5) circle (2pt) node[above right] {$B_5$};
		\end{scope}

	\end{tikzpicture}
	\caption{Magic billiard with flipping over the short axis of the ellipse: each time when the particle hits the boundary, it is reflected with respect to the boundary and immediately magically flipped over the short axis.
	}	\label{fig:magic2}
\end{figure}
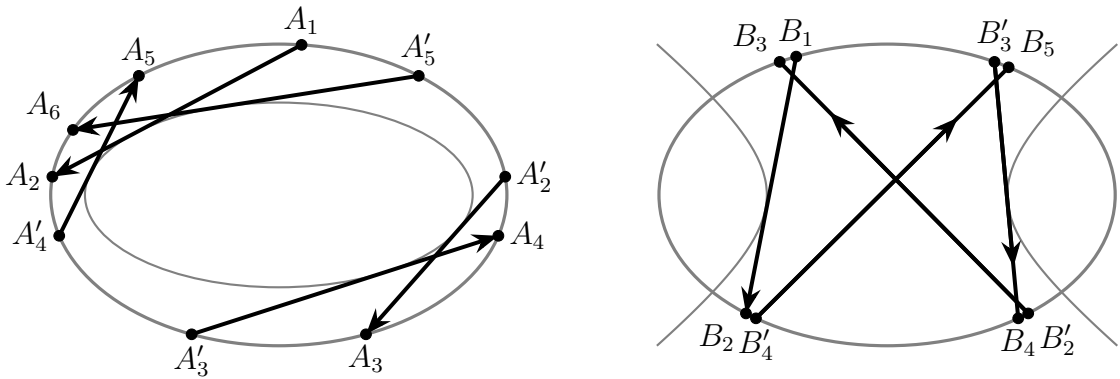

Now, we derive the algebro-geometric condition for periodicity of the such magic billiard.

\begin{theorem}\label{th:divisor-magic2}
In the billiard table bounded by the ellipse $\E$ given by \eqref{eq:ellipse}, consider the magic billiard with flipping over the short axis.
Consider a trajectory of such billiard with the caustic $\C_{\beta}$ from the confocal family \eqref{eq:confocal}.
Such a trajectory
is $n$-periodic if and only if $n$ is even and $nQ_0\sim nQ_{\infty}$.
	
Here $Q_0$ and $Q_{\infty}$ are as in Theorem \ref{th:divisor-magic1}.
\end{theorem}
\begin{proof}
We will use notation as in Theorem \ref{th:divisor-magic1} and its proof.
In the same way as explained there, we can derive the divisor condition \eqref{eq:divisor-cayley}.
	
	In addition to that condition, we need the following:
	\begin{itemize}
		\item the number of times that the particle crossed the long axis must be even: that number equals $m$ if the caustic is ellipse, or $n$ if it is hyperbola;
		\item the sum $n+m$ must be even, since that is the sum of the number of the flippings over the short axis and the number of times when the billiard particle actually crossed that axis.
	\end{itemize}
	In any case, we get that both $m$ and $n$ must be even, so the conditions for closure reduce to $n(Q_0-Q_{c_1})\sim0$, which is equivalent to the stated relation.
\end{proof}

\begin{remark}
It is interesting to note that in this case there are no odd-periodic trajectories.
Notice that, as announced in Remark \ref{rem:even}, the conditions for trajectories of even period are the same as in the case of flipping over the long axis, see Theorem \ref{th:divisor-magic1}, and as in the case of standard billiard, see \cite{DR2019}*{Theorem 2}.
\end{remark}

As a consequence, the analytic and polynomial conditions can be obtained as follows.

\begin{corollary}\label{cor:cayley-magic2}
In the billiard table bounded by the ellipse $\E$ given by \eqref{eq:ellipse}, consider the magic billiard with flipping over the short axis.
Consider a trajectory of such billiard with the caustic $\C_{\beta}$ from the confocal family \eqref{eq:confocal}.
Such a trajectory
is $n$-periodic if and only if $n$ is even and the following equivalent conditions are true:
\begin{itemize}
	\item relation \eqref{eq:cayley-even} is satisfied;
	\item there are polynomials satisfying the polynomial equation \eqref{eq:pol1}.
\end{itemize} 
\end{corollary}

We use Fomenko graph to give a topological description of the system.

\begin{theorem}\label{th:fomenko-magic2}
The Liouville equivalence class of the magic billiard with flipping over the long axis is given by the Fomenko graph in Figure \ref{fig:fom-magic2}.
\end{theorem}
\begin{figure}[h]
	\centering
	\begin{tikzpicture}
		\tikzset{vertex/.style = {shape=circle,draw,minimum size=1.5em}}
		\node[vertex] (aa) at  (0,0) {$\mathbf{A}$};
		\node[vertex] (bb) at  (4,0) {$\mathbf{A^{**}}$};
		\node[vertex] (ee) at  (8,0) {$\mathbf{A}$};
		\path[-{Stealth[length=3mm]}] 
		(bb) edge node[above]{$r=0$} 
		node[below]{$\varepsilon=1$} (aa) 
		(bb) edge node[above]{$r=0$} 
		node[below]{$\varepsilon=1$}(ee); 
		\draw[black, dashed]
	(bb) circle (30pt);
	\node at (4,1.5) {$n=0$};

		\draw (0,-1.5)--(8,-1.5);
		\foreach \x in {0,4,8}
		\draw[shift={(\x,-1.5)},color=black] (0pt,3pt) -- (0pt,-3pt);
		\node[below] at (0,-1.5) { $\lambda=0$};
		\node[below] at (4,-1.5) { $\lambda=b$};
		\node[below] at (8,-1.5) { $\lambda=a$};
	\end{tikzpicture}
	\caption{Theorem \ref{th:fomenko-magic2}: Fomenko graph for the magic billiard with flipping over the short axis.}
	\label{fig:fom-magic2}
\end{figure}
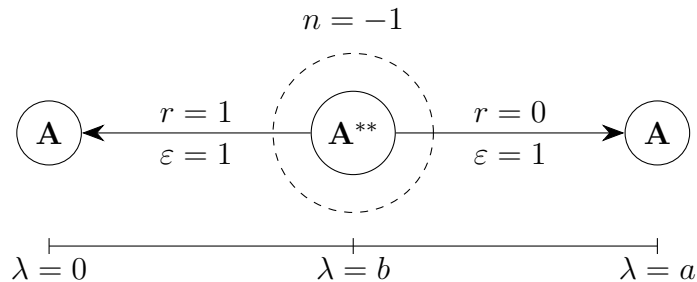
\begin{proof}
Similarly as in the proof of Theorem \ref{th:fomenko-magic1}, the configuration space is homeomorphic to a sphere, where the boundary of the billiard table is represented by an arc.
	
In the phase space, any level set corresponding to the trajectories having a fixed non-degenerate conic as caustic consists of a single Liouville torus.
		
The level set corresponding to the caustic $\E=\C_0$  corresponds to the Fomenko atom $\mathbf{A}$, similarly as in the proof of Theorem \ref{th:fomenko-magic1}.
	
The level set with the caustic $\C_a$ corresponds to the motion back and forth along the short axis of the ellipse $\E$, thus it corresponds to the Fomenko atom of type $\mathbf{A}$.
	
The level set with the caustic $\C_b$ consists of the trajectories that contain the foci of the ellipse $\E$.
First, we will show that the level set is of complexity $2$, i.e.~that it contains exactly two closed orbits.
One such orbit corresponds to the motion from left to the right along the long axis: when the particle hits the boundary, it is magically flipped to the other endpoint of the diameter, from where it continues its motion to the right.
The second closed trajectory corresponds to the motion from right to the left along the long axis.
For any other trajectory on that level set, each segment will pass through the same focus, so there are two separatrices on that level set.

The classification of Fomenko atoms of complexity $2$ from \cite{BF2004} determines that only one of them, $\mathbf{A}^{**}$, has two separatrices.
We note that this is consistent with the numbers of connected components of non-degenerate level sets which are close to $\C_b$: only $1$ for $\lambda<b$ and also $1$ for $\lambda>b$.
The details on calculation of the numerical invariants can be found in \cite{Fokicheva2015}, the domain $\Delta_{\beta}(A_1)^2_{2y}$.
\end{proof}

\subsection{Elliptic billiard with half-turn around the center}

In this case, with $\varphi$ is the half-turn about the center of the ellipse.
Such billiards are a case of the billiards with slipping, see \cite{FVZ2021}.

Examples of its trajectories are depicted in Figure \ref{fig:magic3}.
\begin{figure}[h]
	\centering
	\begin{tikzpicture}[>=Stealth]
		\coordinate (A1) at (0.2995,1.99001);
		\coordinate (A2) at (-2.97804,0.241574);
		\coordinate (A2') at (2.97804,-0.241574);
		\coordinate (A3) at (1.14631,1.84824);
		\coordinate (A3') at (-1.14631,-1.84824);
		\coordinate (A4) at (2.88779,-0.541882);
		\coordinate (A4') at (-2.88779,0.541882);
		\coordinate (A5) at (-1.84042,-1.57943);
		\coordinate (A5') at (1.84042,1.57943);
		\coordinate (A6) at (-2.70702,0.862046);
		\coordinate (A6') at (2.70702,-0.862046);
		
			\coordinate (A2a) at ($ (A1)!8.0/10!(A2) $);

		\draw [very thick,color=gray] (0.,0.) ellipse (3 and 2);
		\draw [thick,color=gray] (0.,0.) ellipse (2.54951 and 1.22474);

		\draw [->,line width=1.5pt] (A1)-- (A2a);
		\draw [line width=1.5pt] (A1)-- (A2);		
		\draw [->,line width=1.5pt] (A2')--(A3);
		\draw [->,line width=1.5pt] (A3')--(A4);
		\draw [->,line width=1.5pt] (A4')--(A5);
		\draw [->,line width=1.5pt] (A5')--(A6);
		
		\draw[black, fill=black]
		(A1) circle (2pt) node[above] {$A_1$};
		\draw[black, fill=black]
		(A2) circle (2pt) node[below left] {$A_2$};
		\draw[black, fill=black]
		(A2') circle (2pt) node[right] {$A_2'$};
		\draw[black, fill=black]
		(A3) circle (2pt) node[above] {$A_3$};
		\draw[black, fill=black]
		(A3') circle (2pt) node[below] {$A_3'$};
		\draw[black, fill=black]
		(A4) circle (2pt) node[below right] {$A_4$};
		\draw[black, fill=black]
		(A4') circle (2pt) node[left] {$A_4'$};
		\draw[black, fill=black]
		(A5) circle (2pt) node[below] {$A_5$};
		\draw[black, fill=black]
		(A5') circle (2pt) node[above] {$A_5'$};
		\draw[black, fill=black]
		(A6) circle (2pt) node[above left] {$A_6$};
		
		\begin{scope}[shift={(8,0)}]
			
			\coordinate (B1) at (-1.19522, 1.83442);
			\coordinate (B2) at (-1.85594, -1.57134);
			\coordinate (B2') at (1.85594, 1.57134);
			\coordinate (B3) at (-1.41241, -1.76448);
			\coordinate (B3') at (1.41241, 1.76448);
			\coordinate (B4) at (1.72475, -1.63643);
			\coordinate (B4') at (-1.72475, 1.63643);
			\coordinate (B5) at (1.59857, -1.69241);
			
			\coordinate (B3a) at ($ (B2')!8.0/10!(B3) $);
			\coordinate (B4a) at ($ (B3')!8.0/10!(B4) $);
			\coordinate (B5a) at ($ (B4')!8.0/10!(B5) $);

			\draw [very thick,color=gray] (0.,0.) ellipse (3 and 2);
			
			\draw[thick, gray, domain=-2:2,smooth,variable=\y]
			plot ({sqrt(2.5+5/3*\y*\y)},{\y});
			\draw[thick, gray, domain=-2:2,smooth,variable=\y]
			plot ({-sqrt(2.5+5/3*\y*\y)},{\y});

			\draw [->,line width=1.5pt] (B1)-- (B2);
			\draw [->,line width=1.5pt] (B2')--(B3);
			\draw [->,line width=1.5pt] (B3')--(B4);
			\draw [->,line width=1.5pt] (B4')--(B5);
			
			\draw[black, fill=black]
			(B1) circle (2pt) node[above] {$B_1$};
			\draw[black, fill=black]	(B2) circle (2pt) node[below left] {$B_2$};
			\draw[black, fill=black](B2') circle (2pt) node[above right] {$B_2'$};
			\draw[black, fill=black](B3) circle (2pt) node[below] {$B_3$};
			\draw[black, fill=black](B3') circle (2pt) node[above] {$B_3'$};
			\draw[black, fill=black](B4) circle (2pt) node[below right] {$B_4$};
			\draw[black, fill=black](B4') circle (2pt) node[above left] {$B_4'$};
			\draw[black, fill=black](B5) circle (2pt) node[below] {$B_5$};
		\end{scope}

	\end{tikzpicture}
	\caption{Magic billiard with slipping by half-ellipse: each time when the particle hits the boundary, it is reflected with respect to the boundary and immediately magically transported to the diametrically opposite point.
	}	\label{fig:magic3}
\end{figure}
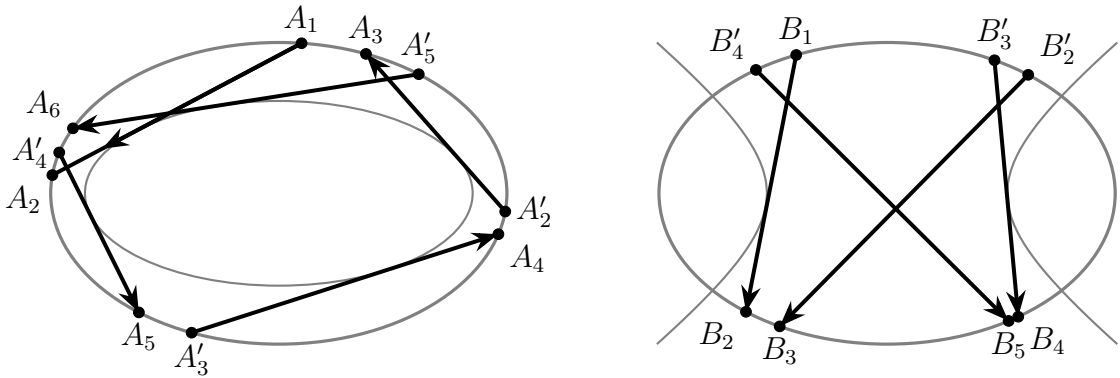

As in the previous cases, we will first derive the divisor conditions for periodicity.

\begin{theorem}\label{th:divisor-magic3}
In the billiard table bounded by the ellipse $\E$ given by \eqref{eq:ellipse}, consider the magic billiard with half-turn about the center.
Consider a trajectory of such billiard with the caustic $\C_{\beta}$ from the confocal family \eqref{eq:confocal}.
Such a trajectory
is $n$-periodic if and only if $nQ_0\sim nQ_{\infty}$.
Here $Q_0$ and $Q_{\infty}$ are as in Theorem \ref{th:divisor-magic1}.
\end{theorem}
\begin{proof}
We follow the notation of Theorem \ref{th:divisor-magic1} and its proof, and, as there, we can derive the divisor condition \eqref{eq:divisor-cayley}.
	
	In addition to that condition, we need that the number of times the particle hit the boundary, the number of times it crossed the long axis, and the number of times it crossed the short axis, are all odd or all even.
	This is equivalent to $m$ and $n$ being both odd or both even.
	
	Thus, the periodicity condition for even $n$ is equivalent to $n(Q_0-Q_{c_1})\sim0$, and for odd $n$ to $n(Q_0-Q_{c_1})+(Q_{c_2}-Q_a)\sim0$.
	Using $2Q_{\infty}\sim2Q_{a}\sim2Q_b\sim2Q_{\beta}$ and $3Q_{\infty}\sim Q_a+Q_b+Q_{\beta}$, we get that the periodicity condition for any $n$ is $nQ_0\sim nQ_{\infty}$.
\end{proof}

From divisor conditions, one can derive the analytic conditions of Cayley's type and the polynomial conditions, similarly as in \cite{DR2019}.
We present them in the following two corollaries.

\begin{corollary}\label{cor:cayley-magic3}
	In the billiard table bounded by the ellipse $\E$ given by \eqref{eq:ellipse}, consider the magic billiard with the half-turn about the center.
	Consider a trajectory of such billiard with the caustic $\C_{\beta}$ from the confocal family \eqref{eq:confocal}.
	Such a trajectory
	is $n$-periodic if and only if:
	\begin{itemize}
		\item $n$ is even and \eqref{eq:cayley-even} is satisfied; or
		\item $n$ is odd and
$$
	\left|	\begin{array}{llll}
		B_2 & B_3 & \dots & B_{(n+1)/2}
		\\
		B_3 & B_4 & \dots & B_{(n+1)/2+1}
		\\
		\dots\\
		B_{(n+1)/2} & B_{(n+1)/2+1} & \dots & B_{n-1}
	\end{array}
	\right|=0.
$$
	\end{itemize} 
The coefficients $B_2$, $B_3$, \dots are as in the statement of Corollary \ref{cor:cayley-magic1}.
\end{corollary}

\begin{corollary}\label{cor:magic-pol3}
	The trajectories of the magic billiard with flipping over the long axis with caustic $\C_{\beta}$ are $n$-periodic if and only if there exists a pair of real polynomials $p_{d_1}$, $q_{d_2}$ of degrees $d_1$, $d_2$ respectively, and satisfying the following:
	\begin{itemize}
		\item[(a)] if $n=2m$ is even,  then $d_1=m$, $d_2=m-2$, and \eqref{eq:pol1};
		\item[(b)] if $n=2m+1$ is odd, then $d_1=m$, $d_2=m-1$, and
		$$
		sp_m^2(s)
		-
		\left(s-\frac1a\right)\left(s-\frac1b\right)\left(s-\frac1{\beta}\right)q_{m-1}^2(s)=1.
		$$
	\end{itemize}	
\end{corollary}

Finally, we give the topological description of this system.

\begin{theorem}[\cite{FVZ2021}]\label{th:fomenko-magic3}
	The Liouville equivalence class of the magic billiard with flipping through the center is given by the Fomenko graph in Figure \ref{fig:fom-magic3}.
\end{theorem}
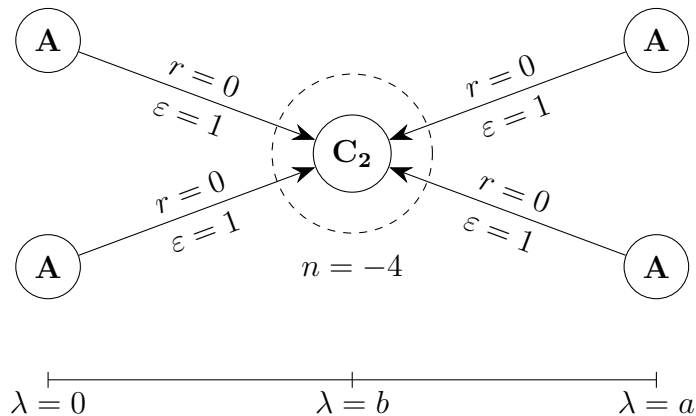
\begin{figure}[h]
	\centering
	\begin{tikzpicture}[>=stealth',el/.style = {inner sep=3pt, align=left, sloped},em/.style = {inner sep=3pt, pos=0.75, sloped}]
		\tikzset{vertex/.style = {shape=circle,draw,minimum size=1.5em}}
		\node[vertex] (aa) at  (0,1.5) {$\mathbf{A}$};
		\node[vertex] (bb) at  (4,0) {$\mathbf{C_2}$};
		\node[vertex] (ee) at  (8,1.5) {$\mathbf{A}$};
		\node[vertex] (ff) at  (0,-1.5) {$\mathbf{A}$};
		\node[vertex] (jj) at  (8,-1.5) {$\mathbf{A}$};
		\path[{Stealth[length=3mm]}-] 
		(bb) edge node[above,sloped]{$r=0$} 
		node[below,sloped]{$\varepsilon=1$} (aa) 
		(bb) edge node[above,sloped]{$r=0$} 
		node[below,sloped]{$\varepsilon=1$}(ee) 
		(bb) edge node[above,sloped]{$r=0$} 
		node[below,sloped]{$\varepsilon=1$}(ff) 
		(bb) edge node[above,sloped]{$r=0$} 
		node[below,sloped]{$\varepsilon=1$}(jj); 
		
			\draw[black, dashed]
		(bb) circle (30pt);
		\node at (4,-1.5) {$n=-4$};

		\draw (0,-3) -- (8,-3);
		\foreach \x in {0,4,8}
		\draw[shift={(\x,-3)},color=black] (0pt,3pt) -- (0pt,-3pt);
		\node[below] at (0,-3) { $\lambda=0$};
		\node[below] at (4,-3) { $\lambda=b$};
		\node[below] at (8,-3) { $\lambda=a$};
	\end{tikzpicture}
	\caption{Theorem \ref{th:fomenko-magic3}: Fomenko graph for the magic billiard with slipping.}
	\label{fig:fom-magic3}
\end{figure}
\begin{proof}
According to Remark \ref{rem:configuration}, we define the configuration space as the billiard table where the points of $\E$ which are diametrically symmetric to each other are identified.
Such configuration space is homeomorphic to the projective plane.
	
In the phase space, any level set corresponding to a any given ellipse as caustic consists of two Liouville tori: one torus contains the trajectories that are winding in the clockwise direction about the caustic, the other torus contains the counterclockwise trajectories.

The level set corresponding to hyperbola as caustic also consists of two Liouville tori: one torus where all segments of the trajectories point downwards, and the other where they point upwards.
	
The level set corresponding to the caustic $\E=\C_0$  corresponds to a pair of Fomenko atoms $\mathbf{A}$: each of them contains a single closed trajectory along the boundary ellipse, winding in one of the two directions.
	
The level set with the caustic $\C_a$ corresponds to the motion the short axis of the ellipse $\E$.
It consists of two closed trajectories -- the upwards and the downwards one.
Thus we have two Fomenko atoms of type $\mathbf{A}$ there.
	
The level set with the caustic $\C_b$ consists of the trajectories that contain the foci of the ellipse $\E$.
Two of those trajectories are closed: one corresponds to motion from left to the right along the long axis, and the other one to the motion from right to the left, so $\C_b$ is of complexity $2$.
	
For any other trajectory on that level set, each segment will pass through the same focus, and each segment points in the same direction: upwards or downwards.
Moreover, such a trajectory is heteroclinic, i.e.~it approaches one of the closed orbits as $t\to+\infty$ and the other one as $t\to-\infty$.
Namely, the trajectories containing the left focus will approach the closed orbit where the motion is from right to the left along the axis in the forward direction and the other closed orbit in the backward direction of time, and vice versa for the trajectories through the right focus.

Thus, there are four separatrices on that level set, all consisting of heteroclinic trajectories.	

Now, following the classification of Fomenko atoms from \cite{BF2004}, we find that the following ones are of complexity $2$ with four separatrices: $\mathbf{C_1}$, $\mathbf{C_2}$, $\mathbf{D_1}$, $\mathbf{D_2}$.
The last two atoms have two homoclinic and two heteroclinic separatrices, thus they do not correspond to this case.
The atom $\mathbf{C_1}$ also can be eliminated, since level sets close to that one consist of only one Liouville torus. 
Since non-degenerate level sets close to $\C_b$ have two connected components for both cases $\lambda<b$ and $\lambda>b$, we conclude that
 it is represented by the Fomenko atom $\mathbf{C}_2$.
\end{proof}

\begin{remark}\label{rem:halfturn}
For a more detailed proof of Theorem \ref{th:fomenko-magic3}, together with the calculation of the numerical invariants, see \cite{FVZ2021}.
\end{remark}

\section{Magic billiards in elliptic annulus}\label{sec:magic-ann}

In this section, we discuss magic billiards in the annulus between two confocal ellipses $\E_1$ and $\E_2$, see Figure \ref{fig:annulus}.
	\begin{figure}[h]
	\centering
	\begin{tikzpicture}[line cap=round,line join=round,>=stealth,x=1.0cm,y=1.0cm]
		\draw [line width=1.5pt,color=black,pattern=north east lines,pattern color=gray!40] (0.,0.) ellipse (3.cm and 2.cm);
		\draw [line width=1.5pt, 
		color=gray, fill=white, fill opacity=1.0] (0.,0.) ellipse (2.45cm and 1.cm);
		\draw[color=black] (-1.8,2.0) node {$\E_{1}$};
		\draw[color=black] (-1.8,0.25) node {$\E_{2}$};
	\end{tikzpicture}
	\caption{The boundary of the billiard table consists of two confocal ellipses.}	\label{fig:annulus}
\end{figure}
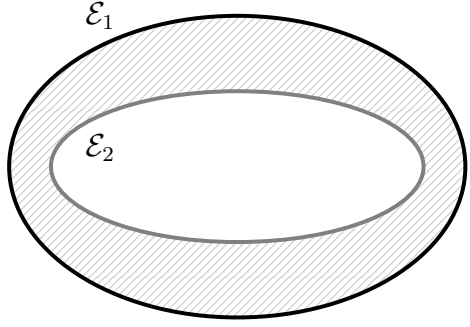
In order to preserve integrability and caustics, we will assume that the mapping $\varphi$ is defined as follows:
\begin{itemize}
	\item the restriction of $\varphi$ to $\E_1$ is the reflection with respect to one of the axes or the central symmetry;
	\item the restriction to $\E_2$ is the identity.
\end{itemize}
The mapping $\varphi^*$ will be defined on the velocity vectors at the points of $\E_1$ in the same way as explained in Section \ref{sec:magic-ell}.
On $\E_2$, mapping $\varphi^*$ is the ordinary billiard reflection.

\begin{remark}
If magic reflection is introduced on the inner boundary, then the dynamics does not any more depend continuously on the initial conditions.
Namely, consider motion parallel to a tangent line to $\E_2$, close to the point of tangency.
If the particle does not cross $\E_2$, the motion is continued straight along that line.
If the particle reaches $\E_2$, it will be, according to the mapping $\varphi$ magically transported to another point of $\E_2$ and continue motion from there.
Thus, the continuity close to tangency to $\E_2$ is lost.
\end{remark}

\begin{example}
Two trajectories when $\varphi$ on the outer boundary is flipping over the long axis are shown in Figure \ref{fig:magic4}.
\end{example}
\begin{figure}[h]
	\centering
	\begin{tikzpicture}[>=Stealth]
		\coordinate (A1) at (-2.99587, 0.10485);
		\coordinate (A2) at (-2.33704, -0.299511);
		\coordinate (A3) at (-2.31015, -1.27597);
		\coordinate (A3') at (-2.31015, 1.27597);
		\coordinate (A4) at (-0.663786, 0.962582);
		\coordinate (A5) at (1.26049, 1.8149);
		\coordinate (A5') at (1.26049, -1.8149);
		\coordinate (A6) at (2.04423, -0.550926);
		\coordinate (A7) at (2.91579, -0.470553);
			\coordinate (A4a) at ($ (A3')!8.0/10!(A4) $);

		\draw [very thick,color=gray] (0.,0.) ellipse (3 and 2);
		\draw [very thick,color=gray] (0.,0.) ellipse (2.44949 and 1);
		\draw [thick,dotted, color=gray] (0.,0.) ellipse (2.34521 and 0.707107);

		\draw [->,line width=1.5pt] (A1)-- (A2);
		\draw [->,line width=1.5pt] (A2)--(A3);
		\draw [line width=1.5pt] (A3')--(A4);
		\draw [->,line width=1.5pt] (A3')--(A4a);
		\draw [->,line width=1.5pt] (A4)--(A5);
		\draw [->,line width=1.5pt] (A5')--(A6);
		\draw [->,line width=1.5pt] (A6)--(A7);

		\draw[black, fill=black]
		(A1) circle (2pt) node[left] {$A_1$};
		\draw[black, fill=black]
		(A2) circle (2pt) node[above right] {$A_2$};
		\draw[black, fill=black]
		(A3) circle (2pt) node[below] {$A_3$};
		\draw[black, fill=black]
		(A3') circle (2pt) node[above left] {$A_3'$};
		\draw[black, fill=black]
		(A4) circle (2pt) node[above] {$A_4$};
		\draw[black, fill=black]
		(A5) circle (2pt) node[below] {$A_5$};
		\draw[black, fill=black]
		(A5') circle (2pt) node[below] {$A_5'$};
		\draw[black, fill=black]
		(A6) circle (2pt) node[above left] {$A_6$};
		\draw[black, fill=black]
		(A7) circle (2pt) node[right] {$A_7$};

		\begin{scope}[shift={(8,0)}]
			\coordinate (B1) at (1.97935, 1.50291);
			\coordinate (B2) at (1.2071, 0.870144);
			\coordinate (B3) at (0.715496, 1.94229);
			\coordinate (B3') at (0.715496, -1.94229);
			\coordinate (B4) at (-0.172605, -0.997514);
			\coordinate (B5) at (-1.08688, -1.86413);
			\coordinate (B5') at (-1.08688, 1.86413);
			\coordinate (B6) at (-1.417, 0.815692);
			\coordinate (B7) at (-2.12313, 1.41301);

			\draw [very thick,color=gray] (0.,0.) ellipse (3 and 2);
			\draw [very thick,color=gray] (0.,0.) ellipse (2.44949 and 1);

			\draw[thick, gray, dotted, domain=-2:2,smooth,variable=\y]
			plot ({sqrt(3+3/2*\y*\y)},{\y});
			\draw[thick, gray, dotted, domain=-2:2,smooth,variable=\y]
			plot ({-sqrt(3+3/2*\y*\y)},{\y});

			\draw [->,line width=1.5pt] (B1)-- (B2);
			\draw [->,line width=1.5pt] (B2)--(B3);
			\draw [->,line width=1.5pt] (B3')--(B4);
			\draw [->,line width=1.5pt] (B4)--(B5);
			\draw [->,line width=1.5pt] (B5')--(B6);
			\draw [->,line width=1.5pt] (B6)--(B7);

			\draw[black, fill=black]
			(B1) circle (2pt) node[above] {$B_1$};
			\draw[black, fill=black]	(B2) circle (2pt) node[below left] {$B_2$};
			\draw[black, fill=black](B3) circle (2pt) node[above] {$B_3$};
			\draw[black, fill=black](B3') circle (2pt) node[below] {$B_3'$};
			\draw[black, fill=black](B4) circle (2pt) node[above] {$B_4$};
			\draw[black, fill=black](B5) circle (2pt) node[below] {$B_5$};
			\draw[black, fill=black](B5') circle (2pt) node[above] {$B_5'$};
			\draw[black, fill=black](B6) circle (2pt) node[below] {$B_6$};
			\draw[black, fill=black](B7) circle (2pt) node[above] {$B_7$};
			
		\end{scope}

	\end{tikzpicture}
	\caption{Billiard between two ellipses with magic flipping over the long axis when the particle hits the outer boundary. The dotted curves are caustics.
	}	\label{fig:magic4}
\end{figure}
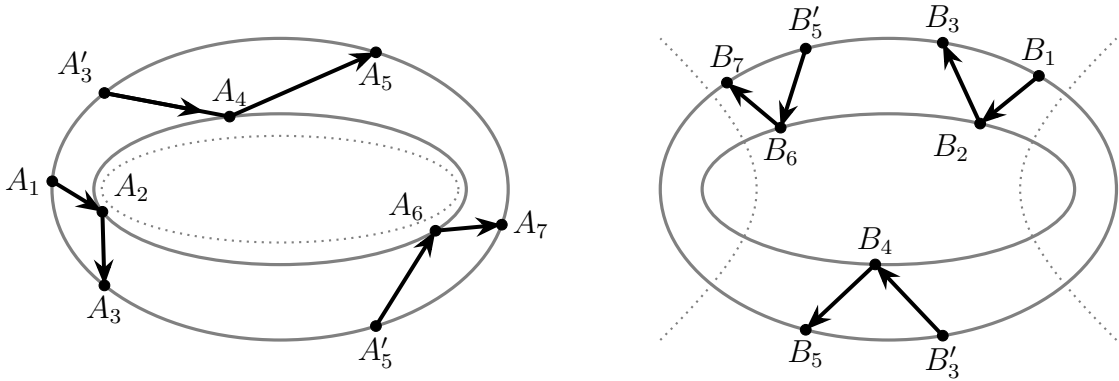

\begin{theorem}
The magic billiard in the annulus between two confocal ellipses with flipping over the long axis on the outer boundary is roughly Liouville equivalent to the magic billiard within an ellipse with flipping over the short axis.
\end{theorem}
\begin{proof}
In the phase space, any level set corresponding to the trajectories having a fixed non-degenerate conic as caustic consists of a single Liouville torus.
	
The level set corresponding to the caustic $\E=\C_0$  corresponds to the Fomenko atom $\mathbf{A}$, similarly as in the proof of Theorem \ref{th:fomenko-magic1}.
	
The level set with the caustic $\C_a$ corresponds to the motion along two segments of the short axis which are within the annulus.
Each time when the particle hits the outer boundary, it is magically flipped to the other segment, thus this level set consists of a single closed trajectory, so we have the Fomenko atom of type $\mathbf{A}$ there.
	
The level set with the caustic $\C_b$ consists of the trajectories that contain the foci of the ellipse $\E$.
Two of those trajectories is closed: each trajectory corresponds to the motion along one of the segments of the long axis which are within the annulus.
For any other trajectory on that level set, the extensions of the segments will pass alternately through the two foci.
In one class of those trajectories, the segments that point from the outer boundary to the inner one contain the left focus, and their limit as time goes to $+\infty$ is the left segment on the long axis, while the time limit to $-\infty$ is the right segment.
In the other class, everything is opposite.
Thus, there are two separatrices on that level set, so it is represented by the Fomenko atom $\mathbf{A}^{**}$.
\end{proof}

\begin{example}
	Two trajectories when $\varphi$ on the outer boundary is flipping over the short axis are shown in Figure \ref{fig:magic5}.
\end{example}
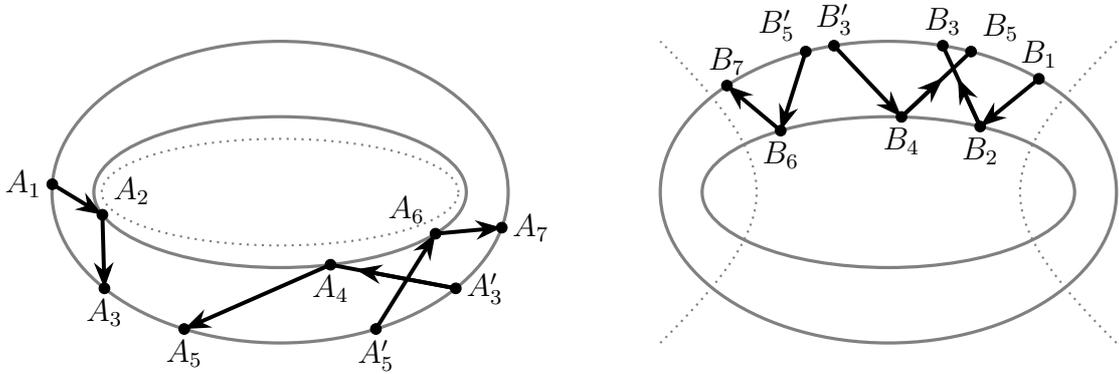
\begin{figure}[h]
	\centering
	\begin{tikzpicture}[>=Stealth]
		\coordinate (A1) at (-2.99587, 0.10485);
		\coordinate (A2) at (-2.33704, -0.299511);
		\coordinate (A3) at (-2.31015, -1.27597);
		\coordinate (A3') at (2.31015, -1.27597);
		\coordinate (A4) at (0.663786, -0.962582);
		\coordinate (A5) at (-1.26049, -1.8149);
		\coordinate (A5') at (1.26049, -1.8149);
		\coordinate (A6) at (2.04423, -0.550926);
		\coordinate (A7) at (2.91579, -0.470553);
			\coordinate (A4a) at ($ (A3')!8.0/10!(A4) $);

		\draw [very thick,color=gray] (0.,0.) ellipse (3 and 2);
		\draw [very thick,color=gray] (0.,0.) ellipse (2.44949 and 1);
		\draw [thick,dotted, color=gray] (0.,0.) ellipse (2.34521 and 0.707107);

		\draw [->,line width=1.5pt] (A1)-- (A2);
		\draw [->,line width=1.5pt] (A2)--(A3);
		\draw [line width=1.5pt] (A3')--(A4);
		\draw [->,line width=1.5pt] (A3')--(A4a);
		\draw [->,line width=1.5pt] (A4)--(A5);
		\draw [->,line width=1.5pt] (A5')--(A6);
		\draw [->,line width=1.5pt] (A6)--(A7);

		\draw[black, fill=black]
		(A1) circle (2pt) node[left] {$A_1$};
		\draw[black, fill=black]
		(A2) circle (2pt) node[above right] {$A_2$};
		\draw[black, fill=black]
		(A3) circle (2pt) node[below] {$A_3$};
		\draw[black, fill=black]
		(A3') circle (2pt) node[right] {$A_3'$};
		\draw[black, fill=black]
		(A4) circle (2pt) node[below] {$A_4$};
		\draw[black, fill=black]
		(A5) circle (2pt) node[below] {$A_5$};
		\draw[black, fill=black]
		(A5') circle (2pt) node[below] {$A_5'$};
		\draw[black, fill=black]
		(A6) circle (2pt) node[above left] {$A_6$};
		\draw[black, fill=black]
		(A7) circle (2pt) node[right] {$A_7$};

		\begin{scope}[shift={(8,0)}]
			\coordinate (B1) at (1.97935, 1.50291);
			\coordinate (B2) at (1.2071, 0.870144);
			\coordinate (B3) at (0.715496, 1.94229);
			\coordinate (B3') at (-0.715496, 1.94229);
			\coordinate (B4) at (0.172605, 0.997514);
			\coordinate (B5) at (1.08688, 1.86413);
			\coordinate (B5') at (-1.08688, 1.86413);
			\coordinate (B6) at (-1.417, 0.815692);
			\coordinate (B7) at (-2.12313, 1.41301);
			
			\coordinate (B5a) at ($ (B4)!6.0/10!(B5) $);
			\coordinate (B3a) at ($ (B2)!6.0/10!(B3) $);

			\draw [very thick,color=gray] (0.,0.) ellipse (3 and 2);
			\draw [very thick,color=gray] (0.,0.) ellipse (2.44949 and 1);

			\draw[thick, gray, dotted, domain=-2:2,smooth,variable=\y]
			plot ({sqrt(3+3/2*\y*\y)},{\y});
			\draw[thick, gray, dotted, domain=-2:2,smooth,variable=\y]
			plot ({-sqrt(3+3/2*\y*\y)},{\y});

			\draw [->,line width=1.5pt] (B1)-- (B2);
			\draw [->,line width=1.5pt] (B2)--(B3a);
			\draw [line width=1.5pt] (B2)--(B3);
			\draw [->,line width=1.5pt] (B3')--(B4);
			\draw [->,line width=1.5pt] (B4)--(B5a);
			\draw [line width=1.5pt] (B4)--(B5);
			\draw [->,line width=1.5pt] (B5')--(B6);
			\draw [->,line width=1.5pt] (B6)--(B7);

			\draw[black, fill=black]
			(B1) circle (2pt) node[above] {$B_1$};
			\draw[black, fill=black]	(B2) circle (2pt) node[below] {$B_2$};
			\draw[black, fill=black](B3) circle (2pt) node[above] {$B_3$};
			\draw[black, fill=black](B3') circle (2pt) node[above] {$B_3'$};
			\draw[black, fill=black](B4) circle (2pt) node[below] {$B_4$};
			\draw[black, fill=black](B5) circle (2pt) node[above right] {$B_5$};
			\draw[black, fill=black](B5') circle (2pt) node[above left] {$B_5'$};
			\draw[black, fill=black](B6) circle (2pt) node[below] {$B_6$};
			\draw[black, fill=black](B7) circle (2pt) node[above] {$B_7$};
			
		\end{scope}

	\end{tikzpicture}
	\caption{Billiard between two ellipses with magic flipping over the short axis when the particle hits the outer boundary. The dotted curves are caustics.
	}	\label{fig:magic5}
\end{figure}

\begin{theorem}
	The magic billiard in the annulus between two confocal ellipses with flipping over the short axis on the outer boundary is roughly Liouville equivalent to the magic billiard within an ellipse with flipping over the long axis.
\end{theorem}
\begin{proof}
	In the phase space, any level set corresponding to the trajectories having an ellipse as caustic consists of a single Liouville torus.
	On the other hand, if the caustic is hyperbola, there are two tori: each corresponding to one connected component within the annulus between the two branches of the hyperbola.
	
	The level sets corresponding to the caustic $\E=\C_0$  corresponds to the Fomenko atom $\mathbf{A}$.

	The level set with the caustic $\C_a$ corresponds to the motion along two segments of the short axis which are within the annulus.
	Each segment is covered by one closed trajectory on that level set, thus both of them corresponds to the Fomenko atom $\mathbf{A}$.
	
	The level set with the caustic $\C_b$ consists of the trajectories that contain the foci of the ellipse $\E$.
	One of those trajectories is closed and it is placed on the long axis, traversing the two segments.
	Any other trajectory on that level set is placed on one side of the long axis, thus there are two separatrices.
We conclude that this level set corresponds to the the Fomenko atom $\mathbf{B}$.
\end{proof}

\begin{example}
	Two trajectories when $\varphi$ on the outer boundary is slipping by half-ellipse are shown in Figure \ref{fig:magic6}.
\end{example}
\begin{figure}[h]
	\centering
	\begin{tikzpicture}[>=Stealth]
		\coordinate (A1) at (-2.99587, 0.10485);
		\coordinate (A2) at (-2.33704, -0.299511);
		\coordinate (A3) at (-2.31015, -1.27597);
		\coordinate (A3') at (2.31015, 1.27597);
		\coordinate (A4) at (0.663786, 0.962582);
		\coordinate (A5) at (-1.26049, 1.8149);
		\coordinate (A5') at (1.26049, -1.8149);
		\coordinate (A6) at (2.04423, -0.550926);
		\coordinate (A7) at (2.91579, -0.470553);
		\coordinate (A4a) at ($ (A3')!8.0/10!(A4) $);

		\draw [very thick,color=gray] (0.,0.) ellipse (3 and 2);
		\draw [very thick,color=gray] (0.,0.) ellipse (2.44949 and 1);
		\draw [thick,dotted, color=gray] (0.,0.) ellipse (2.34521 and 0.707107);

		\draw [->,line width=1.5pt] (A1)-- (A2);
		\draw [->,line width=1.5pt] (A2)--(A3);
		\draw [line width=1.5pt] (A3')--(A4);
		\draw [->,line width=1.5pt] (A3')--(A4a);
		\draw [->,line width=1.5pt] (A4)--(A5);
		\draw [->,line width=1.5pt] (A5')--(A6);
		\draw [->,line width=1.5pt] (A6)--(A7);

		\draw[black, fill=black]
		(A1) circle (2pt) node[left] {$A_1$};
		\draw[black, fill=black]
		(A2) circle (2pt) node[above right] {$A_2$};
		\draw[black, fill=black]
		(A3) circle (2pt) node[below] {$A_3$};
		\draw[black, fill=black]
		(A3') circle (2pt) node[right] {$A_3'$};
		\draw[black, fill=black]
		(A4) circle (2pt) node[above] {$A_4$};
		\draw[black, fill=black]
		(A5) circle (2pt) node[above] {$A_5$};
		\draw[black, fill=black]
		(A5') circle (2pt) node[below] {$A_5'$};
		\draw[black, fill=black]
		(A6) circle (2pt) node[above left] {$A_6$};
		\draw[black, fill=black]
		(A7) circle (2pt) node[right] {$A_7$};

		\begin{scope}[shift={(8,0)}]
			\coordinate (B1) at (1.97935, 1.50291);
			\coordinate (B2) at (1.2071, 0.870144);
			\coordinate (B3) at (0.715496, 1.94229);
			\coordinate (B3') at (-0.715496, -1.94229);
			\coordinate (B4) at (0.172605, -0.997514);
			\coordinate (B5) at (1.08688, -1.86413);
			\coordinate (B5') at (-1.08688, 1.86413);
			\coordinate (B6) at (-1.417, 0.815692);
			\coordinate (B7) at (-2.12313, 1.41301);

			\draw [very thick,color=gray] (0.,0.) ellipse (3 and 2);
			\draw [very thick,color=gray] (0.,0.) ellipse (2.44949 and 1);

			\draw[thick, gray, dotted, domain=-2:2,smooth,variable=\y]
			plot ({sqrt(3+3/2*\y*\y)},{\y});
			\draw[thick, gray, dotted, domain=-2:2,smooth,variable=\y]
			plot ({-sqrt(3+3/2*\y*\y)},{\y});

			\draw [->,line width=1.5pt] (B1)-- (B2);
			\draw [->,line width=1.5pt] (B2)--(B3);
			\draw [->,line width=1.5pt] (B3')--(B4);
			\draw [->,line width=1.5pt] (B4)--(B5);
			\draw [->,line width=1.5pt] (B5')--(B6);
			\draw [->,line width=1.5pt] (B6)--(B7);

			\draw[black, fill=black]
			(B1) circle (2pt) node[above] {$B_1$};
			\draw[black, fill=black]	(B2) circle (2pt) node[below] {$B_2$};
			\draw[black, fill=black](B3) circle (2pt) node[above] {$B_3$};
			\draw[black, fill=black](B3') circle (2pt) node[below] {$B_3'$};
			\draw[black, fill=black](B4) circle (2pt) node[above] {$B_4$};
			\draw[black, fill=black](B5) circle (2pt) node[below] {$B_5$};
			\draw[black, fill=black](B5') circle (2pt) node[above left] {$B_5'$};
			\draw[black, fill=black](B6) circle (2pt) node[below] {$B_6$};
			\draw[black, fill=black](B7) circle (2pt) node[above] {$B_7$};
			
		\end{scope}

	\end{tikzpicture}
	\caption{Billiard between two ellipses with slipping by half-ellipse along the outer boundary. The dotted curves are caustics.
	}	\label{fig:magic6}
\end{figure}
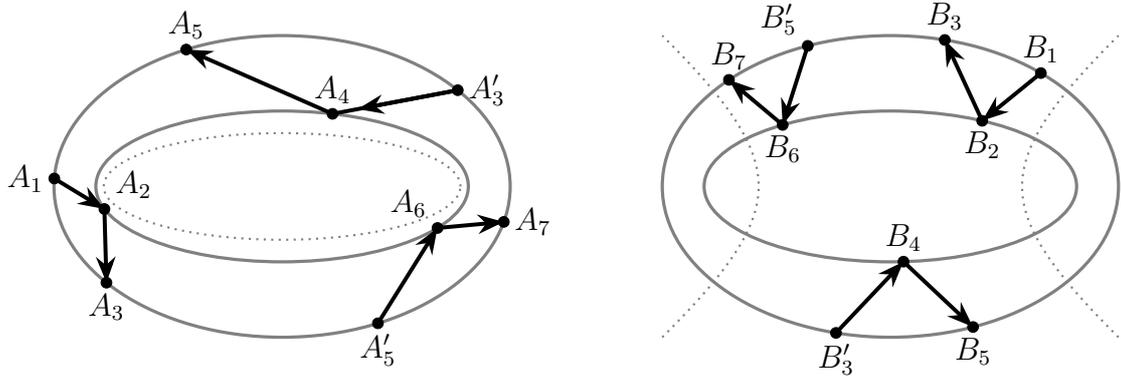

\begin{theorem}\label{th:fom-ann}
The Liouville equivalence class of the magic billiard in an elliptic annulus with slipping by half-ellipse along the outer boundary is given by the Fomenko graph in Figure \ref{fig:fom-ann}.
\end{theorem}
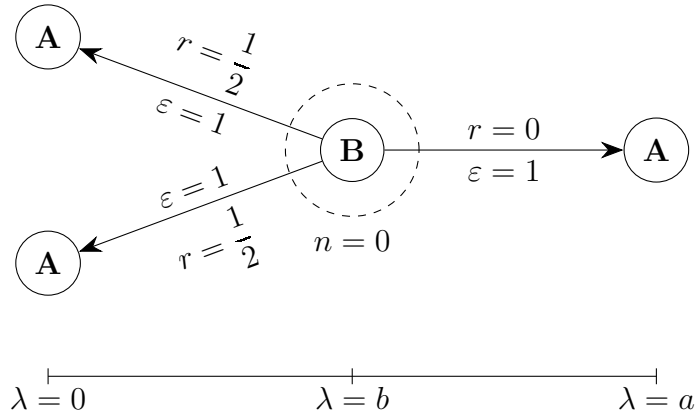
\begin{figure}[h]
	\centering
	\begin{tikzpicture}
		\tikzset{vertex/.style = {shape=circle,draw,minimum size=1.5em}}
		\node[vertex] (aa) at  (8,0) {$\mathbf{A}$};
		\node[vertex] (bb) at  (4,0) {$\mathbf{B}$};
		\node[vertex] (ee) at  (0,1.5) {$\mathbf{A}$};
		\node[vertex] (jj) at  (0,-1.5) {$\mathbf{A}$};
		\path[-{Stealth[length=3mm]}] 
		(bb) edge node[above]{$r=\infty$} 
		node[below]{$\varepsilon=1$} (aa) 
		(bb) edge node[above,sloped]{$r=\dfrac12$} 
		node[below,sloped]{$\varepsilon=1$} (ee) 
		(bb) edge node[below,sloped]{$r=\dfrac12$} 
		node[above,sloped]{$\varepsilon=1$} (jj); 
		

		\draw (0,-3) -- (8,-3);
		\foreach \x in {0,4,8}
		\draw[shift={(\x,-3)},color=black] (0pt,3pt) -- (0pt,-3pt);
		\node[below] at (0,-3) { $\lambda=0$};
		\node[below] at (4,-3) { $\lambda=b$};
		\node[below] at (8,-3) { $\lambda=a$};
	\end{tikzpicture}
	\caption{Theorem \ref{th:fom-ann}: Fomenko graph for the magic billiard in the annulus with flipping over the long axis.}
	\label{fig:fom-ann}
\end{figure}
\begin{proof}
Any level set corresponding to the trajectories having a fixed ellipse as caustic consists of two Liouville tori: each torus contains the trajectories the winding in one direction around the annulus. 
	On the other hand, if the caustic is hyperbola, there is only one torus.
	
	The level set corresponding to the caustic $\E=\C_0$ contains only two closed trajectories: each is winding in one direction along the boundary.
	Thus, we have two Fomenko atoms of type $\mathbf{A}$ there.
	
	The level set with the caustic $\C_a$ corresponds to the motion along two segments of the short axis which are within the annulus.
	There is only one trajectory on that level set, so it corresponds to the Fomenko atom $\mathbf{A}$.
	
	The level set with the caustic $\C_b$ consists of the trajectories that contain the foci of the ellipse $\E$.
	One of those trajectories is closed and it is placed on the long axis, traversing the two segments.
There are two separatrices on that level set: one separatrix contains the trajectories where the motion is above the long axis on the left and below the axis on the right as $t\to+\infty$, and vice versa for $t\to-\infty$.
The opposite holds for the other separatrix.
We can conclude that the corresponding Fomenko atom is $\mathbf{B}$.
\end{proof}

\begin{remark}
The example of billiard within elliptic annulus with slipping along the outer ellipse was analyzed in \cite{FVZ2021}.
There, the configuration space was defined differently: by gluing two identical annuli along inner and the outer boundaries, so the resulting Fomenko graph there is different.
\end{remark}

\begin{remark}
The billiards within an annulus with flipping over one of the axes on the outer boundary correspond to the topological billiards $\Delta_{\beta}(2B_2'')_{xx}$ and $\Delta_{\beta}(2B_1)_{yy}$ from \cite{Fokicheva2015}.
Together with Remarks \ref{rem:flip} and \ref{rem:halfturn}, and the end of the proof of Theorem \ref{th:fomenko-magic2}, this exemplifies the significance of the Fomenko conjecture, since the systems in Sections \ref{sec:magic-ell} and \ref{sec:magic-ann} are constructed to be integrable.
\end{remark}

\begin{remark}
Various classes of glued sets and billiard systems on them were considered in \cite{Kudr2015}.
\end{remark}

\section{Conclusions and discussion}\label{sec:conclusions}

We note that the magic billiards, as introduced in Section \ref{sec:definition} represents a very broad class, where various subclasses may be of interest for exciting future research.

In particular, integrable cases for magic billiards within ellipse are not exhausted by the list considered in this paper. It would be interesting to explore a more general class of magic billiards $(\E,\varphi,\varphi^*)$, including those which preserve caustic. One natural generalization includes studying domains with piece-wise smooth boundaries. 

We note that billiard ordered games introduced in \cite{DR2004} and studied further in \cite{DGR2022} can also be generalised using magic reflections.

\section*{Acknowledgements} 
This paper is devoted to Academician Anatoly Timofeevich Fomenko on the occasion of his 80-th anniversary.  The authors have been glad to learn a lot from personal contacts with Anatoly Timofeevich and his distinguished students, from their lectures, books, and papers. The authors wish Anatoly Timofeevich many happy returns and further success in his scientific work and the work of his scientific school. 

This research is supported by the Discovery Project No.~DP190101838
\emph{Billiards within confocal quadrics and beyond} from the Australian Research Council, 
by  Mathematical Institute of the Serbian Academy of Sciences and Arts, the Science Fund of Serbia grant Integrability and Extremal Problems in Mechanics, Geometry and
Combinatorics, MEGIC, Grant No.~7744592 and the Ministry for Education, Science, and Technological Development of Serbia and the Simons Foundation grant no.~854861.

The authors are grateful to the anonymous referees for careful reading and comments, which helped improve and clarify the exposition.
M.~R.~is grateful to Professors John Roberts and Wolfgang Schief and the School of Mathematics and Statistics of the UNSW Sydney for their hospitality and support during final stages of work on this paper, and to Max Planck Institute for Mathematics in Bonn for its hospitality and financial support.

\bibliographystyle{amsalpha}
\bibliography{References}

\begin{bibdiv}
\begin{biblist}

\bib{ADR2019}{article}{
      author={Adabrah, Anani~Komla},
      author={Dragovi\'c, Vladimir},
      author={Radnovi\'c, Milena},
       title={Periodic billiards within conics in the {M}inkowski plane and {A}khiezer polynomials},
        date={2019},
     journal={Regul. Chaotic Dyn.},
      volume={24},
      number={5},
       pages={464\ndash 501},
}

\bib{Ar}{book}{
      author={Arnold, V.~I.},
       title={Mathematical methods of classical mechanics},
   publisher={Springer},
        date={1989},
}

\bib{BBM2010}{article}{
      author={Bolsinov, A.~V.},
      author={Borisov, A.~V.},
      author={Mamaev, I.~S.},
       title={Topology and stability of integrable systems},
        date={2010},
     journal={Russian Mathematical Surveys},
      volume={65},
      number={2},
       pages={259\ndash 318},
}

\bib{BF2004}{book}{
      author={Bolsinov, A.V.},
      author={Fomenko, A.T.},
       title={Integrable {H}amiltonian systems: Geometry, topology, classification},
   publisher={CRC Press},
        date={2004},
}

\bib{BF2024}{article}{
      author={Belozerov, G.~V.},
      author={Fomenko, A.~T.},
       title={Rotation functions of integrable billiards as orbital invariants},
        date={2024},
     journal={Dokl. Math.},
      volume={109},
      number={1},
       pages={1\ndash 5},
}

\bib{Bir}{book}{
      author={Birkhoff, G.},
       title={Dynamical {Systems}},
      series={American {Mathematical} {Society} / {Providence}, {RI}},
   publisher={American Mathematical Society},
        date={1927},
}

\bib{BMF1990}{article}{
      author={Bolsinov, A.~V.},
      author={Matveev, S.~V.},
      author={Fomenko, A.~T.},
       title={Topological classification of integrable {H}amiltonian systems with two degrees of freedom. {L}ist of systems of small complexity},
        date={1990},
     journal={Russian Mathematical Surveys},
      volume={45},
      number={2},
       pages={59\ndash 94},
}

\bib{Bol1990}{article}{
      author={Bolotin, S.~V.},
       title={Integrable {B}irkhoff billiards},
        date={1990},
     journal={Vestnik Moskov. Univ. Ser. I Mat. Mekh.},
      number={2},
       pages={33\ndash 36, 105},
}

\bib{DGKh2024}{article}{
      author={Drivas, Theodore~D.},
      author={Glukhovskiy, Daniil},
      author={Khesin, Boris},
       title={Pensive billiards, point vortices, and pucks},
        date={2024},
     journal={arXiv:2408.03279 [math.DS]},
}

\bib{DGR2022}{article}{
      author={Dragovi\'c, Vladimir},
      author={Gasiorek, Sean},
      author={Radnovi\'c, Milena},
       title={Billiard ordered games and books},
        date={2022},
     journal={Regul. Chaotic Dyn.},
      volume={27},
      number={2},
       pages={132\ndash 150},
}

\bib{DGR2021}{article}{
      author={Dragovi\'c, Vladimir},
      author={Gasiorek, Sean},
      author={Radnovi\'c, Milena},
       title={Integrable billiards on a {M}inkowski hyperboloid: Extremal polynomials and topology},
        date={2022},
     journal={Sbornik Mathematics},
      volume={213},
      number={9},
       pages={3\ndash 38},
}

\bib{DR2004}{article}{
      author={Dragovi\'c, Vladimir},
      author={Radnovi\'c, Milena},
       title={Cayley-type conditions for billiards within k quadrics in $\mathbf{R}^d$},
        date={2004},
     journal={Journal of Physics A: Mathematical and General},
      volume={37},
      number={4},
       pages={1269\ndash 1276},
}

\bib{DR2009}{article}{
      author={Dragovi\'c, Vladimir},
      author={Radnovi\'c, Milena},
       title={Bifurcations of {L}iouville tori in elliptical billiards},
        date={2009},
     journal={Regular and Chaotic Dynamics},
      volume={14},
      number={4},
       pages={479\ndash 494},
}

\bib{DR2010}{article}{
      author={Dragovi\'c, Vladimir},
      author={Radnovi\'c, Milena},
       title={Integrable billiards and quadrics},
        date={2010},
     journal={Russian Mathematical Surveys},
      volume={65},
      number={2},
       pages={319\ndash 379},
}

\bib{DR2011}{book}{
      author={Dragovi\'c, Vladimir},
      author={Radnovi\'c, Milena},
       title={Poncelet porisms and beyond: Integrable billiards, hyperelliptic {J}acobians and pencils of quadrics},
      series={Frontiers in mathematics},
   publisher={Birkh\"{a}user/Springer},
        date={2011},
}

\bib{DR2012}{article}{
      author={Dragovi\'c, Vladimir},
      author={Radnovi\'c, Milena},
       title={Ellipsoidal billiards in pseudo-{Euclidean} spaces and relativistic quadrics},
        date={2012-10},
     journal={Advances in Mathematics},
      volume={231},
      number={3-4},
       pages={1173\ndash 1201},
}

\bib{DR2013}{article}{
      author={Dragovi\'c, Vladimir},
      author={Radnovi\'c, Milena},
       title={Minkowski plane, confocal conics, and billiards},
        date={2013},
     journal={Publ Inst Math (Belgr)},
      volume={94},
      number={108},
       pages={17\ndash 30},
}

\bib{DR2017}{article}{
      author={Dragovi\'c, Vladimir},
      author={Radnovi\'c, Milena},
       title={Topological invariants for elliptical billiards and geodesics on ellipsoids in the {M}inkowski space},
        date={2017},
     journal={Journal of Mathematical Sciences},
      volume={223},
      number={6},
       pages={686\ndash 694},
}

\bib{DR2019}{article}{
      author={Dragovi\'c, Vladimir},
      author={Radnovi\'c, Milena},
       title={Caustics of {P}oncelet polygons and classical extremal polynomials},
        date={2019},
     journal={Regul. Chaotic Dyn.},
      volume={24},
      number={1},
       pages={1\ndash 35},
}

\bib{DR2018}{article}{
      author={Dragovi\'c, Vladimir},
      author={Radnovi\'c, Milena},
       title={Periodic ellipsoidal billiard trajectories and extremal polynomials},
        date={2019},
     journal={Comm. Math. Phys.},
      volume={372},
      number={1},
       pages={183\ndash 211},
}

\bib{FKK2020}{article}{
      author={Fomenko, A.},
      author={Kharcheva, I.},
      author={Kibkalo, V.},
       title={{R}ealization of integrable {H}amiltonian systems by billiard books},
        date={2021},
     journal={Trans. Moscow Math. Soc.},
      volume={82},
       pages={37\ndash 64},
}

\bib{Fokicheva2014}{article}{
      author={Fokicheva, V.~V.},
       title={Classification of billiard motions in domains bounded by confocal parabolas},
        date={2014},
        ISSN={0368-8666},
     journal={Mat. Sb.},
      volume={205},
      number={8},
       pages={139\ndash 160},
}

\bib{Fokicheva2015}{article}{
      author={Fokicheva, V.~V.},
       title={Topological classification of billiards in locally planar domains bounded by arcs of confocal quadrics},
        date={2015},
     journal={Mat. Sb.},
      volume={206},
      number={10},
       pages={127\ndash 176},
}

\bib{Fomenko1987}{article}{
      author={Fomenko, A.~T.},
       title={The topology of syrfaces of constant enerly in integrable {H}amiltonian systems, and abostructions to integrability},
        date={1987},
     journal={Izv. Akad. Nauk SSSR Ser. Mat.},
      volume={29},
      number={3},
       pages={629\ndash 658},
}

\bib{Fom2023}{article}{
      author={Fomenko, A.~T.},
       title={Billiards of variable configuration and billiards with slipping in {H}amiltonian geometry and topology},
        date={2023},
     journal={Lobachevskii J. Math.},
      volume={44},
      number={10},
       pages={4512\ndash 4522},
}

\bib{FV2019a}{article}{
      author={Fomenko, A.~T.},
      author={Vedyushkina, V.~V.},
       title={Implementation of integrable systems by topological, geodesic billiards with potential and magnetic field},
        date={2019},
     journal={Russian Journal of Mathematical Physics},
      volume={26},
      number={3},
       pages={320\ndash 333},
}

\bib{FV2019}{article}{
      author={Fomenko, A.~T.},
      author={Vedyushkina, V.~V.},
       title={Singularities of integrable {L}iouville systems, reduction of integrals to lower degree and topological billiards: recent results},
        date={2019},
     journal={Theoretical and Applied Mechanics},
      volume={46},
      number={1},
       pages={47\ndash 63},
}

\bib{FV2021}{article}{
      author={Fomenko, A.~T.},
      author={Vedyushkina, V.~V.},
       title={Billiards with changing geometry and their connection with the implementation of the {Z}hukovsky and {K}ovalevskaya cases},
        date={2021},
     journal={Russ. J. Math. Phys.},
      volume={28},
      number={3},
       pages={317\ndash 332},
}

\bib{FoVed}{article}{
      author={Fomenko, Anatoly},
      author={Vedyushkina, Viktoria},
       title={Billiards and intregrable systems},
        date={2023},
     journal={Russian Mathematical Surveys},
      volume={78},
      number={5},
       pages={881\ndash 954},
}

\bib{FVZ2021}{article}{
      author={Fomenko, A.~T.},
      author={Vedyushkina, V.~V.},
      author={Zav'yalov, V.~N.},
       title={Liouville foliations of topological billiards with slipping},
        date={2021},
     journal={Russ. J. Math. Phys.},
      volume={28},
      number={1},
       pages={37\ndash 55},
}

\bib{FZ1991}{article}{
      author={Fomenko, A.~T.},
      author={Zieschang, H.},
       title={A topological invariant and a criterion for the equivalence of integrable {H}amiltonian systems with two degrees of freedom},
        date={1991},
     journal={Izv. Akad. Nauk SSSR Ser. Mat.},
      volume={36},
      number={3},
       pages={567\ndash 596},
}

\bib{GM2024}{article}{
      author={Glutsyuk, Alexey},
      author={Matveev, Vladimir~S.},
       title={If a {M}inkowski billiard is projective, it is the standard billiard},
        date={2024},
     journal={arXiv:2407.20159 [math.DS]},
}

\bib{GT2002}{article}{
      author={Gutkin, Eugene},
      author={Tabachnikov, Serge},
       title={Billiards in {F}insler and {M}inkowski geometries},
        date={2002},
     journal={J. Geom. Phys.},
      volume={40},
      number={3-4},
       pages={277\ndash 301},
}

\bib{KT1991}{book}{
      author={Kozlov, V.~V.},
      author={Treshch\"{e}v, D.~V.},
       title={Billiards},
      series={Translations of Mathematical Monographs},
   publisher={American Mathematical Society, Providence, RI},
        date={1991},
      volume={89},
        note={A genetic introduction to the dynamics of systems with impacts, Translated from the Russian by J. R. Schulenberger},
}

\bib{KT2009}{article}{
      author={Khesin, Boris},
      author={Tabachnikov, Serge},
       title={Pseudo-{R}iemannian geodesics and billiards},
        date={2009},
     journal={Adv. Math.},
      volume={221},
      number={4},
       pages={1364\ndash 1396},
}

\bib{Kudr2015}{article}{
      author={Kudryavtseva, E.~A.},
       title={Liouville integrable generalized billiard flows and theorems of {P}oncelet type},
        date={2015},
     journal={Fundam. Prikl. Mat.},
      volume={20},
      number={3},
       pages={113\ndash 152},
}

\bib{PRK2020}{article}{
      author={Pnueli, M.},
      author={Rom-Kedar, V.},
       title={On the structure of {Hamiltonian} impact systems},
        date={2021},
     journal={Nonlinearity},
      volume={34},
      number={4},
       pages={2611\ndash 2658},
}

\bib{Radn2003}{article}{
      author={Radnovi\'c, Milena},
       title={A note on billiard systems in {F}insler plane with elliptic indicatrices},
        date={2003},
     journal={Publ. Inst. Math. (Beograd) (N.S.)},
      volume={74(88)},
       pages={97\ndash 101},
}

\bib{R2015}{article}{
      author={Radnovi\'c, M.},
       title={Topology of the elliptical billiard with the {H}ooke's potential},
        date={2015},
     journal={Theoretical and Applied Mechanics},
      volume={42},
      number={1},
       pages={1\ndash 9},
}

\bib{RRK2008}{article}{
      author={Radnovi\'{c}, M.},
      author={Rom-Kedar, V.},
       title={Foliations of isonergy surfaces and singularities of curves},
        date={2008Dec},
     journal={Regular and Chaotic Dynamics},
      volume={13},
      number={6},
       pages={645\ndash 668},
}

\bib{Tab1997}{article}{
      author={Tabachnikov, Serge},
       title={Introducing projective billiards},
        date={1997},
     journal={Ergodic Theory Dynam. Systems},
      volume={17},
      number={4},
       pages={957\ndash 976},
}

\bib{Tab}{book}{
      author={Tabachnikov, S.},
       title={Geometry and {Billiards}},
      series={Student mathematical library},
   publisher={American Mathematical Society},
        date={2005},
}

\bib{VK2018}{article}{
      author={Vedyushkina, V.~V.},
      author={Kharcheva, I.~S.},
       title={Billiard books model all three-dimensional bifurcations of integrable {H}amiltonian systems},
        date={2018},
     journal={Sbornik: Mathematics},
      volume={209},
      number={12},
       pages={1690\ndash 1727},
}

\bib{VZ2022}{article}{
      author={Vedyushkina, V.~V.},
      author={Zav'yalov, V.~N.},
       title={Realization of geodesic flows with a linear first integral by billiards with slipping},
        date={2022},
     journal={Mat. Sb.},
      volume={213},
      number={12},
       pages={31\ndash 52},
}

\bib{Zav2023}{article}{
      author={Zav'yalov, V.~N.},
       title={Billiard with slipping by an arbitrary rational angle},
        date={2023},
     journal={Mat. Sb.},
      volume={214},
      number={9},
       pages={3\ndash 26},
}

\end{biblist}
\end{bibdiv}

\end{document}